\documentclass[final,onefignum,onetabnum]{siamart171218}

\usepackage{lipsum}
\usepackage{amsfonts}
\usepackage{graphicx}
\usepackage{epstopdf}
\usepackage{algorithmic}
\ifpdf
  \DeclareGraphicsExtensions{.eps,.pdf,.png,.jpg}
\else
  \DeclareGraphicsExtensions{.eps}
\fi

\PassOptionsToPackage{hidelinks}{hyperref}
\usepackage{booktabs}       
\usepackage{amsfonts}       
\usepackage{amssymb}
\usepackage{mathtools}
\usepackage{nicefrac}       
\usepackage{microtype}      
\usepackage{amsmath}
\usepackage{xspace}
\usepackage{siunitx}
\usepackage[T1]{fontenc}
\usepackage{dsfont}
\usepackage{bm}
\usepackage{pifont}
\usepackage[export]{adjustbox}
\usepackage{relsize}
\usepackage{wrapfig}
\usepackage{xcolor,colortbl}
\usepackage{hhline}
\usepackage{makecell}
\usepackage{physics}
\usepackage{bbm}
\usepackage{pdfpages}
\usepackage{changes}
\usepackage{tabularx}
\usepackage[nameinlink]{cleveref}
\usepackage{changes}
\usepackage{float,lscape}
\usepackage{tikz}
\usepackage{color,soul}
\usepackage{multirow}
\usepackage{braket}
\usepackage{cite}

\def\eg{\emph{e.g.}}
\def\ie{\emph{i.e.}}
\def\ii{\mathrm{i}}
\definecolor{masoncolor}{rgb}{0.98, 0.27, 0.62}

\definecolor{tabred}{RGB}{180, 30, 30}      
\definecolor{tabblue}{RGB}{25, 90, 150}     
\definecolor{tabgreen}{RGB}{30, 120, 30}    
\definecolor{taborange}{RGB}{170, 70, 10}   

\newsiamremark{remark}{Remark}
\newsiamremark{hypothesis}{Hypothesis}
\crefname{hypothesis}{Hypothesis}{Hypotheses}
\newsiamthm{claim}{Claim}

\headers{PF theory
for networks with complex weights}{Y. Tian, M. A. Porter, and L. B\"ottcher}
\title{Generalizing Perron--Frobenius theory and eigenvector-based centralities to networks with complex edge weights\thanks{Submitted to the editors DATE.
\funding{YT and LB acknowledge financial support from hessian.AI.}}}

\author{Yu Tian\thanks{Center for Systems Biology Dresden, 01307 Dresden, Germany; 
Max Planck Institute for the Physics of Complex Systems, 01187 Dresden, Germany;
Max-Planck Institute of Molecular Cell Biology and Genetics,
01307 Dresden, Germany; Cluster of Excellence, Physics of Life, TU Dresden, 01307 Dresden, Germany 
  (\email{yu.tian.research@gmail.com}).}
\and Mason A.\ Porter\thanks{Department of Mathematics, University of California, Los Angeles, CA, 90095, United States of
America; Department of Sociology, University of California, Los Angeles, CA, 90095, United States of
America; Sante Fe Institute, Santa Fe, NM, 87501, United States of America
  (\email{mason@math.ucla.edu}).} \and Lucas B\"ottcher\thanks{Department of Computational Science and Philosophy, Frankfurt School of Finance and Management, 60322, Frankfurt am Main, Germany and Dept.\ of Medicine, University of Florida, Gainesville, FL, 32610, United States of
America
  (\email{l.boettcher@fs.de}).}}

\usepackage{amsopn}

\makeatletter
\newcommand*{\addFileDependency}[1]{
  \typeout{(#1)}
  \@addtofilelist{#1}
  \IfFileExists{#1}{}{\typeout{No file #1.}}
}
\makeatother

\ifpdf
\hypersetup{
  pdftitle={Generalizing Perron--Frobenius theory and eigenvector-based centralities to networks with complex edge weights},
  pdfauthor={Y.\ Tian, M.\ A.\ Porter, and L.\ B\"ottcher}
}
\fi

\begin{document}

\maketitle

\begin{abstract}
    A fundamental concept in linear algebra and its applications to network analysis is the Perron--Frobenius (PF) theorem, which underpins eigenvector-based centrality measures such as eigenvector centrality, PageRank, and hubs and authorities.
    By invoking the PF theorem, we know for strongly connected networks with positive edge weights that the eigenvector corresponding to the largest eigenvalue of the weight matrix yields a well-defined centrality measure (namely, eigenvector centrality).
    Traditional formulations of the PF theorem and associated centrality measures assume that networks have real-valued weights. However, many networks in areas such as quantum information, quantum chemistry, electrodynamics, and machine learning have complex-valued edge weights. In this paper, we study generalizations of the PF theorem to complex-valued matrices, establish connections between these generalizations, and propose generalized eigenvector-based centrality measures to analyzing node importances in networks with complex edge weights. 
    We also prove results about the existence of complex-weighted networks that satisfy generalized PF properties and calculate associated centrality measures for several examples, which we draw from application areas such as electron transport, circuit analysis, mathematical chemistry, and communication networks.
\end{abstract}



%
\section{Introduction}
The theory of non-negative matrices is fundamental to the study of networks. It forms the basis for methods like the computation of centrality measures, which have been used in a wide range of disciplines~\cite{pillai2005perron,estrada2010network,gleich2015pagerank,newman2018networks}, including sociology~\cite{katz1953new,bonacich1972a,bonacich1972factoring,freeman1977set,freeman1978centrality,bonacich1987power,bonacich2007some}, computer science~\cite{brin1998}, biology~\cite{lohmann2010eigenvector,negre2018eigenvector}, economics~\cite{schweitzer2009economic,demirer2018estimating}, electrical engineering~\cite{wang2010electrical}, and even sports~\cite{keener1993perron,callaghan2007,wasche2017social}. 
A cornerstone of centrality analysis is \emph{Perron--Frobenius (PF) theory}, which is based on Oskar Perron's foundational work on positive matrices~\cite{perron1907} and the generalization of Perron's work by Georg Ferdinand Frobenius to non-negative matrices~\cite{frobenius1912}. See~\cite{hawkins2008continued} for a detailed account of the history of PF theory.

The traditional PF theorem\footnote{In the present paper, we use the term ``traditional PF theorem'' to refer to the version of the PF theorem for positive matrices.} asserts that a square matrix $\mathbf{A}$ of dimension $N\times N$ with positive entries $a_{ij}$ (with $i,j\in\{1,\ldots,N\}$) has an associated positive real number $r(\mathbf{A}) = \lambda_1$, which is called the \emph{PF eigenvalue}. The PF eigenvalue is simple (\ie, it has an algebraic multiplicity of one) and has a strictly larger magnitude than any other eigenvalue $\lambda_i$ (with $i \in \{2,\ldots,N\}$) of $\mathbf{A}$. That is, $\lambda_1 > |\lambda_i|$ for $i \in \{2,\ldots,N\}$. Therefore, the PF eigenvalue equals the spectral radius
\begin{equation}
    \rho(\mathbf{A}) = \max_{\lambda\in\sigma(\mathbf{A})} |\lambda|\,, 
\label{eq:spectral_radius}
\end{equation}
where $\sigma(\mathbf{A}) = \{\lambda_i\}_{i \in \{1, \ldots, N\}}$ is the set of eigenvalues (\ie, the spectrum) of $\mathbf{A}$. The entries of the corresponding \emph{PF eigenvector} are all positive.\footnote{The PF eigenvalue and PF eigenvector are often referred to as the Perron eigenvalue and Perron eigenvector in the literature.} A PF eigenvalue and its corresponding PF eigenvector also exist for non-negative, irreducible matrices. When one interprets $\mathbf{A}$ as the adjacency matrix of a directed network, the network is strongly connected (\ie, any node is reachable from any other node) if and only if $\mathbf{A}$ is irreducible.

One can use the PF eigenvectors of positive and irreducible non-negative matrices that constitute adjacency or weight matrices of networks as measures of node importance in these networks. The entries of PF eigenvectors are the eigenvector centralities of the corresponding nodes of a network~\cite{newman2018networks}. A variety of other centrality measures, such as PageRank~\cite{gleich2015pagerank,brin1998} and hubs and authorities~\cite{kleinberg1999authoritative}, arise from PF eigenvectors of other matrices. Researchers have leveraged PF theory to extend such eigenvector-based centralities to temporal and multilayer networks~\cite{taylor2017,taylor2021tunable,taylor2023}. Other centralities, such as Katz centrality~\cite{katz1953new}, depend on eigenvectors of a matrix but are not themselves entries of an eigenvector of a matrix.

For many networked systems (such as transportation networks, supply chains, and social networks), it is natural to consider edges with 
non-negative and real-valued weights. However, other networked systems have both positive and negative weights \cite{diazdiaz2025}.
For example, biological systems can involve both activatory (\ie, positive) and inhibitory (\ie, negative) interactions, political networks can include both trust and mistrust relationships, and economic systems often include both cooperative and antagonistic dynamics~\cite{Altafini_2012_opinion,Facchetti_2011_large,Marvel_2010_signedContDyn,tian_2021_role,tian_2022_thesis}. Motivated by such applications, many generalizations of the PF theorem have been developed for matrices with both positive and negative edges (\ie, signed edges)~\cite{berman1979nonnegative,noutsos2006perron,altafini2014predictable,DBLP:conf/globecom/GaoLB15,DBLP:conf/cdc/ShiAB17,shi2019}. 
However, these extensions are still limited to networks with real-valued weights. Many systems in fields such as quantum information, quantum chemistry, electrodynamics, and machine learning involve networks with complex-valued weights~\cite{bottcher2024complex}, for which one cannot directly use either the traditional PF theorem or its generalizations to matrices with signed edges. In the present paper, we study generalizations of PF theorems to complex-valued matrices, establish connections between these generalizations, and propose generalized eigenvector-based centrality measures to analyze node importances in networks with complex edge weights. We prove results about the existence of complex-weighted networks that satisfy generalized PF properties, and we discuss uniqueness conditions {for eigenvectors}.
\subsection{Related Work}
There have been many advances in PF theory, which illustrate how matrix structure (and tensor structure) influences spectral properties in networks~\cite{gautier2023nonlinear}. We discuss a variety of these contributions.

Several researchers have examined extensions of PF theorems to matrices with both positive and negative entries. Noutsos~\cite{noutsos2006perron} studied PF theorems for eventually non-negative (or eventually positive) matrices $\mathbf{A}$, which are real-valued matrices that become non-negative (or positive) when raised to sufficiently large positive integer powers [\ie, there exists an integer $k_0 > 0$ such that $\mathbf{A}^k\ge 0$ (or $\mathbf{A}^k > 0$) entrywise for all $k \ge k_0$] and showed that such matrices have PF-type properties.\footnote{Noutsos~\cite{noutsos2006perron} considered eventual positivity for real matrices (which can have negative entries). A related but more restrictive concept is that of primitive matrices (\ie, non-negative matrices that become strictly positive when raised to a certain positive integer power). The {traditional} PF theorem also applies to primitive matrices~\cite{horn2012matrix}.} 
A related line of research concerns matrices, such as $M$-matrices and $Z$-matrices (which both have nonpositive off-diagonal entries), with specific sign patterns and structural conditions that yield PF-type spectral features~\cite{berman1979nonnegative}.
Rump introduced the concept of a ``sign-real spectral radius'', which is {not restricted to specific sign patterns and generalizes several key properties of the PF eigenvalue of non-negative matrices to a more general class of real-valued matrices}~\cite{rump1997pf1,rump1998pf2,rump1999pf3,rump2001pf4}. 

Researchers have also extended the traditional PF theorem to complex-valued matrices, which arise in areas such as quantum information~\cite{childs2002example,childs2004spatial,boettcher2021classical}, quantum chemistry~\cite{lekishvili1997characterization,estrada2006atomic}, electrodynamics~\cite{paul2007analysis,strub2019modeling}, and machine learning~\cite{noest1987,leung1991complex,kobayashi2010exceptional,zhang2021optical}. Noutsos and Varga~\cite{noutsos2012perron} generalized the traditional PF theorem to complex-valued matrices
by introducing a complex analogue of eventual non-negativity, and Saxena et al.~\cite{saxena2024flows} subsequently built on those generalizations to characterize a consensus process on networks with complex weights.
Very recently, Saxena et al.~\cite{saxena2025laplacianflowscomplex} further developed complex PF-type criteria by introducing a real-dominance condition for consensus in flows in complex-valued Laplacian matrices~\cite{saxena2025laplacianflowscomplex}.
Rump~\cite{rump2003perron} generalized the Perron root of a non-negative matrix to complex-valued matrices by replacing the standard linear eigenvalue problem with a nonlinear eigenvalue problem to define a ``sign-complex spectral radius''.\footnote{More broadly, PF theory has been extended to complex matrices through the framework of ``complex cones''~\cite{rugh2010cones} and also to matrices {$\mathbf{A}$ with non-negative $\mathbf{A}^k$ for some integer 
$k > 0$~\cite{tudisco2015complex}.}}
The usual framework for nonlinear eigenvalue problems generalizes the standard linear eigenvalue problem by introducing nonlinearity in the eigenvalue parameter while keeping a linear dependence on the eigenvector~\cite{dumont2007solution,effenberger2013robust}. By contrast, Rump's approach has a linear-in-magnitude dependence on the eigenvalue but introduces nonlinearity by taking an entrywise absolute value after the action of a matrix on an eigenvector. In Sections~\ref{sec:generalized_pf} and~\ref{sec:comparison_generalizations}, we further discuss the approaches of Noutsos and Varga~\cite{noutsos2012perron} and Rump~\cite{rump2003perron}.

PF theory has also been generalized to networks with polyadic (\ie, ``higher-order'') interactions. Chang et al.~\cite{chang2008perron,chang2013variational} developed a PF theory for non-negative tensors, and Benson~\cite{benson2019three} later leveraged their approach to develop centrality measures for hypergraphs. In related work, Michoel and Nachtergaele~\cite{michoel2012alignment} proposed a generalization of PF theory to hypergraphs and used it to derive spectral-clustering algorithms. More recently, Gautier, Tudisco, and Hein~\cite{gautier2019unifying,gautier2023nonlinear} developed a unifying nonlinear PF theory for non-negative tensors that subsumed many earlier tensor PF results as special cases. 
\subsection{Organization of Our Paper}
Our paper proceeds as follows. In Section~\ref{sec:complex_weights}, we define networks with complex-valued edge weights and overview generalizations of the PF theorem for such networks. In Section~\ref{sec:existence_results}, we use the concept of ``generalized-switching equivalence'' to establish existence results for families of complex-weighted networks that satisfy generalized PF properties. Based on these results, in Section~\ref{sec:generalized_centralities}, we describe ways to generalize eigenvector centrality and other eigenvector-based centralities (\eg, PageRank, hubs, and authorities) that arise from solutions of eigenvalue problems to networks with complex edge weights. 
In Section~\ref{sec:applications}, we {perform calculations that are associated with}
these generalized PF properties and eigenvector-based centralities in several examples, which we draw from application areas such as electron transport, circuit analysis, mathematical chemistry, and communication networks. In Section~\ref{sec:discussion}, we summarize and discuss our results.
In the Supplemental Material, we examine which complex-weight matrices that arise in the description of social and communication networks~\cite{hoser2005eigenspectral,sanchez_infrastructure} satisfy a certain generalized PF property. 
Our code and {scripts to generate all figures are} available at \url{https://gitlab.com/ComputationalScience/perron-frobenius}.
\section{Networks with Complex Edge Weights}
\label{sec:complex_weights}
In Section~\ref{sec:definitions}, we define networks with complex-valued edge weights. In Section~\ref{sec:generalized_pf}, we overview existing generalizations of the PF theorem for such networks. In Section~\ref{sec:comparison_generalizations}, we compare these generalizations to each other.
\subsection{Definitions}
\label{sec:definitions}
We consider networks in the form of weighted and directed networks (\ie, graphs) $G = (V,E,w)$, where $V$ is a set of nodes, $E$ is a set of edges, and $w\colon E\rightarrow \mathbb{C}$ a function that assigns a complex weight to each edge. The number of nodes is $N = |V|$. We describe weighted edges between nodes using two matrices: (1) an adjacency matrix $\mathbf{A} \in \{0,1\}^{N \times N}$; and (2) a weight matrix (\ie, a weighted adjacency matrix) $\mathbf{W} \in \mathbb{C}^{N\times N}$. The entries $a_{ij}$ of the matrix $\mathbf{A}$ are $1$ if there is a directed edge from node $i$ to node $j$, and $a_{ij} = 0$ otherwise. 
We do not consider self-edges or self-weights, so $a_{ii} = w_{ii} = 0$. To represent complex-valued relationships between nodes, we let the weight-matrix entries $w_{ij} = r_{ij}e^{\ii \varphi_{ij}}$ be complex numbers with magnitude $r_{uv}$ and phase 
$\varphi_{ij}$. When a network is undirected, $a_{ij} = a_{ji}$, $r_{ij} = r_{ji}$, and $\varphi_{ij} = \varphi_{ji}$. We set $w_{ij} = 0$ if and only if $a_{ij} = 0$.
\subsection{Generalizations of the Traditional PF Theorem to Matrices with Complex-Valued Entries}
\label{sec:generalized_pf}
According to the traditional PF theorem~\cite{perron1907,frobenius1912}, the weight matrix $\mathbf{W}$ (and hence also the adjacency matrix $\mathbf{A}$) of a strongly connected network with positive weights has a simple positive eigenvalue (the so-called \emph{PF eigenvalue}, which equals the spectral radius $\rho(\mathbf{W})$) that is strictly larger in magnitude than all other eigenvalues [see Eq.~\eqref{eq:spectral_radius}]. 
The corresponding \emph{PF eigenvector} has only positive entries and is the only eigenvector with only positive entries (and hence it is unique). The PF eigenvector of $\mathbf{W}$ (and the PF eigenvector of $\mathbf{A}$) gives one way to measure the centralities of the nodes of a network~\cite{newman2018networks}.
This notion of centrality is known as ``eigenvector centrality''~\cite{bonacich1972factoring}.

The traditional PF theorem does not carry over directly to matrices with complex edge weights. Therefore, one cannot guarantee that there is a PF eigenvalue and a corresponding eigenvector with all positive entries to use as a centrality measure for all strongly connected complex-valued weight matrices
$\mathbf{W}\in\mathbb{C}^{N\times N}$. However, extensions of the traditional PF theorem~\cite{rump2003perron,noutsos2012perron} offer a potential framework to define eigenvector centrality (and some of its generalizations, such as PageRank~\cite{brezinski2006pagerank}) in networks with certain types of complex weight matrices. 

Noutsos and Varga~\cite{noutsos2012perron} proposed PF properties for complex-valued matrices that are based on the definition~\eqref{eq:spectral_radius} of the spectral radius. In Table~\ref{tab:generalized_PF}, we overview these generalized PF properties.

\begin{table}[htbp]
\footnotesize
\centering
\renewcommand{\arraystretch}{1.5}
\begin{tabular}{>{\raggedright\arraybackslash}p{2.5cm}>{\raggedright\arraybackslash}p{9.5cm}}
\toprule
\textbf{Definition} & \textbf{Description} \\
\hline
PF property & A matrix $\mathbf{W}\in \mathbb{C}^{N\times N}$ has the PF property if it has an eigenvalue $\lambda_1 = \rho(\mathbf{W}) > 0$ with an associated nonzero column eigenvector $\mathbf{x} = (x_1,\ldots,x_N)^\top$ whose entries $x_i$ are all non-negative (\ie, $x_i\geq 0$ for $i \in \{1, \ldots, N\}$). The vector $\mathbf{x}$ is called a \textit{right PF eigenvector}. \\
\hline
Strong PF property & A matrix $\mathbf{W}\in \mathbb{C}^{N\times N}$ has the strong PF property if it has a simple eigenvalue $\lambda_1 = \rho(\mathbf{W}) > 0$ that satisfies $\lambda_1 > |\lambda_j|$ for all other eigenvalues $\lambda_j$ of $\mathbf{W}$, where $j\in\{2,\ldots,N\}$. The corresponding eigenvector $\mathbf{x} = (x_1,\ldots,x_N)^\top$ has all positive entries (\ie, $x_i > 0$ for $i \in \{1, \ldots, N\}$) and is called a \textit{strong right PF eigenvector}. \\
\hline
Complex PF property & A matrix $\mathbf{W}\in \mathbb{C}^{N\times N}$ has the complex PF property if it has a positive dominant eigenvalue $\lambda_1$ with an associated nonzero eigenvector $\mathbf{x} = (x_1,\ldots,x_N)^\top$ whose entries all have non-negative real parts (\ie, $\mathrm{Re}(x_i) \geq 0$ for $i \in \{1, \ldots, N\}$). The vector $\mathbf{x}$ is called a \textit{complex right PF eigenvector}. \\
\hline
Strong complex PF property & A matrix $\mathbf{W}\in \mathbb{C}^{N\times N}$ has the strong complex PF property if it has a positive dominant eigenvalue $\lambda_1$ that is simple and satisfies $\lambda_1 > |\lambda_j|$ for all other eigenvalues $\lambda_j$ of $\mathbf{W}$, where $j\in\{2, \ldots, N\}$). The entries of the corresponding eigenvector $\mathbf{x} = (x_1,\ldots,x_N)^\top$ all have positive real parts (\ie, $\mathrm{Re}(x_i) > 0$ for $i \in \{1, \ldots, N\}$). The vector $\mathbf{x}$ is called a \textit{strong complex right PF eigenvector}. \\
\bottomrule
\end{tabular}
\caption{Summary of the PF properties for complex matrices of Noutsos and Varga~\cite{noutsos2012perron}.}
\label{tab:generalized_PF}
\end{table}

The issue of uniqueness of the eigenvectors that are associated with the generalized PF properties was not considered in~\cite{noutsos2012perron}. For a complex-valued normal weight matrix $\mathbf{W}$ with the strong PF property, the eigenvector with all positive entries is unique. The strong PF property guarantees the existence of a dominant eigenvalue $\lambda_1 > 0$ with a corresponding eigenvector with strictly positive entries. Normal matrices\footnote{{A ``normal'' matrix $\mathbf{W}\in\mathbb{C}^{N\times N}$ satisfies $\mathbf{W} \mathbf{W}^\dagger = \mathbf{W}^\dagger \mathbf{W}$, where $\mathbf{W}^\dagger$ denotes the Hermitian conjugate of $\mathbf{W}$~\cite{strang2012linear}.}} have orthogonal eigenvectors~\cite{strang2012linear}, so no other eigenvector can share this property and the positive eigenvector is thus unique. However, matrices that satisfy the PF property, the complex PF property, or the strong complex PF property can possess multiple eigenvectors with the corresponding property, even when its associated network is strongly connected. In these cases, the eigenvectors can have non-negative entries for the PF property, entries with non-negative real parts for the complex PF property, and entries with positive real parts for the strong complex PF property.

To motivate a generalization of the {traditional} PF theorem by Rump~\cite{rump2003perron}, we rewrite the definition of the spectral radius in Eq.~\eqref{eq:spectral_radius} to explicitly incorporate the definition of an eigenvalue problem. We thus obtain
\begin{equation}
    \rho(\mathbf{W}) = \max \left\{ |\lambda| \colon \mathbf{W} \mathbf{x} = \lambda \mathbf{x}\,, \, \lambda\in\mathbb{C}\,, \,
  \mathbf{x}   \in \mathbb{C}^N\!\setminus \! \{0\} 
    \right\}\,.
    \label{eq:rho_ev}
\end{equation}
For complex-valued matrices $\mathbf{W}\in\mathbb{C}^{N\times N}$, the \emph{sign-complex spectral radius}~\cite{rump2003perron} 
\begin{equation}
    \tilde{\rho}(\mathbf{W}) = \max \left\{ |\lambda| \colon |\mathbf{W} \mathbf{x}| = |\lambda \mathbf{x}| \,\,\, \text{for all} \,\,\, \lambda\in\mathbb{C}\,, \, 
    \mathbf{x} \in \mathbb{C}^N\!\setminus \! \{0\} \right\}
\label{eq:sign_compl_spect_rad}
\end{equation}
generalizes Eq.~\eqref{eq:rho_ev}. In accordance with \cite{rump2003perron}, we use an entrywise interpretation {of} absolute values and comparisons of vectors and matrices. For instance, the expression $\mathbf{x}\geq 0$ signifies that the entries $x_i$ of the vector $\mathbf{x} = (x_1,\ldots,x_N)^\top$ satisfy $x_i \geq 0$ for all $i\in\{1,\ldots,N\}$.\footnote{{Using this convention, the sign-complex spectral radius in Eq.~\eqref{eq:sign_compl_spect_rad} is $\tilde{\rho}(\mathbf{W}) = \max \left\{ |\lambda| \colon \! \left|\sum_{j = 1}^N w_{ij} x_j\right| = |\lambda x_i| \,\,\, \text{for all} \,\,\,  i\in\{1,\ldots,N\}\,, \lambda\in\mathbb{C}\,, \, 
    \mathbf{x} \in \mathbb{C}^N\!\setminus \! \{0\} \right\}$.}} Accordingly, with this notation, we emphasize that $|\lambda \mathbf{x}|$ is a vector.

The sign-complex spectral radius resembles the spectral radius $\rho(\mathbf{W})$ for non-negative matrices $\mathbf{W}$. The spectral radius and the sign-complex spectral radius satisfy the inequality 
$\tilde{\rho}(\mathbf{W})\leq \rho(|\mathbf{W}|)$. That is, the sign-complex spectral radius of $\mathbf{W}$ is less than or equal to the spectral radius of $|\mathbf{W}|$.
Additionally, Rump~\cite{rump2003perron} showed that $\tilde{\rho}(\mathbf{W}) = 1$ if $\mathbf{W}\in\mathbb{C}^{N\times N}$ is unitary (\ie, $\mathbf{W} \mathbf{W}^\dagger=\mathbf{W}^\dagger \mathbf{W} = \mathbf{I}$) and that $\tilde{\rho}(\mathbf{W}) = \rho(\mathbf{W})$ if $\mathbf{W}\in\mathbb{C}^{N\times N}$ is normal.

Equation~\eqref{eq:sign_compl_spect_rad} is based on a nonlinear eigenvalue problem that differs from a standard (\ie, linear) eigenvalue problem. However, instead of involving a power series in~$\lambda$, as is common in many nonlinear eigenvalue problems~\cite{dumont2007solution}, nonlinearity appears in Rump's formulation through the application of the map $F(\mathbf{x})\coloneqq|\mathbf{W}\mathbf{x}|$.
Although there exist cone-theoretic PF methods\footnote{A ``cone-theoretic'' PF method is an approach to study eigenvalue problems for linear~\cite{KreinRutman1948,birkhoff1957extensions} and nonlinear~\cite{nussbaum1987hilbert,lemmens2012nonlinear} maps that act on a cone (\eg, the non-negative orthant) to derive conditions for the existence and, in some cases, the uniqueness of eigenvectors.}
that provide existence and uniqueness conditions for eigenvectors of linear and nonlinear self-maps of a cone, these results do not apply directly to Rump's sign-complex spectral radius~\cite{rump2002variational,rump2003perron}. In particular, one cannot reduce the generalized eigenvalue equation $F(\mathbf{x}) = |\lambda \mathbf{x}|$ to an eigenproblem of the form $F(\mathbf{x}) = \lambda \mathbf{x}$, and $F(\mathbf{x})$ is generally \emph{not} order-preserving, as cancellations in $\mathbf{W}\mathbf{x}$ can cause some entries of $|\mathbf{W}\mathbf{x}|$ to decrease even when the entries of $\mathbf{x}$ increase.

There are several variational characterizations of the sign-complex spectral radius $\tilde{\rho}(\mathbf{W})$~\cite{rump2002variational,rump2003perron}, including nonlinear analogues of the classical Collatz--Wielandt formulations~\cite{rump2002variational}. 
These characterizations express $\tilde{\rho}(\mathbf{W})$ as a Collatz--Wielandt-type max--min extremum of entrywise ratios $|(\mathbf{W}\mathbf{x})_i/\mathbf{x}_i|$ for $\mathbf{x}_i\neq 0$. One maximizes 
$|(\mathbf{W}\mathbf{x})_i/\mathbf{x}_i|$ over all orthants), so there is not a
distinguished orthant for general complex matrices. This situation differs from
the role of the non-negative orthant in traditional PF theory.
These expressions provide a theoretical foundation for the computation of the sign-complex spectral radius, but they typically are not straightforward to evaluate efficiently.
Rump~\cite{rump2002variational} proved that every orthant contains a generalized eigenvector and that the extrema in the variational characterizations of $\tilde{\rho}(\mathbf{W})$ are attained.
However, unlike in traditional PF theory, the sign-complex spectral radius $\tilde{\rho}(\mathbf{W})$ does not necessarily correspond to a {strictly positive} eigenvector. Indeed, Rump~\cite{rump2002variational} gave an example in which the absolute weight matrix $|\mathbf{W}|$ (with the absolute value taken entrywise, as discussed previously) is irreducible, yet no vector $\mathbf{x} > 0$ satisfies $|\mathbf{W} \mathbf{x}| \le \tilde{\rho}(\mathbf{W}) \mathbf{x}$. This example highlights that the irreducibility of $|\mathbf{W}|$ is not sufficient to guarantee that the sign-complex spectral radius has an associated strictly positive eigenvector.

\begin{figure}
    \centering
    \begin{tikzpicture}[>=stealth, thick]

\node[circle, draw, fill=pink!80, minimum size=15pt] (1) at (90:2)  {1};
\node[circle, draw, fill=cyan!40, minimum size=15pt] (2) at (210:2) {2};
\node[circle, draw, fill=yellow!60, minimum size=15pt] (3) at (330:2){3};

\draw[->]
   (1) to[bend left=20] node[pos=.42, left] {$e^{\ii \varphi}$} (2);

\draw[->]
   (2) to[bend left=20] node[pos=.51, below] {$e^{\ii r^2 \varphi}$} (3);

\draw[->]
   (3) to[bend left=20] node[pos=.58, right, xshift=1.3mm] {$e^{\ii
r\varphi}$} (1);

\draw[->]
   (2) to[bend left=40] node[pos=.52, left] {$e^{-\ii \varphi}$} (1);

\draw[->]
   (3) to[bend left=40] node[pos=.45, below] {$e^{-\ii r^2 \varphi}$} (2);

\draw[->]
   (1) to[bend left=40] node[pos=.45, right, xshift=1mm] {$e^{-\ii r
\varphi}$} (3);

\end{tikzpicture}
    \caption{A complete and directed triangle network with complex edge weights.
    }
    \label{fig:three_node}
  \end{figure}
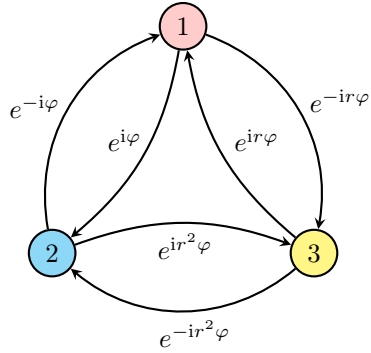  

Computing the sign-complex spectral radius\textemdash including identifying suitable generalized eigenvectors\textemdash is usually computationally demanding because it is necessary to solve variational optimization problems in many orthants, so we do not consider this generalization of the spectral radius in our subsequent theorems and examples. Instead, we restrict our attention to the PF generalizations in Table~\ref{tab:generalized_PF}, as proposed by Noutsos and Varga~\cite{noutsos2012perron}. These generalizations are based on the standard spectral radius and are thus compatible with standard analytical and numerical eigenvalue methods (such as diagonalization and power iteration)~\cite{burden2015numerical}.

We now present an example to clarify key definitions. Consider the Hermitian weight matrix
\begin{equation}
    \mathbf{W} = \begin{pmatrix}
		0 & e^{\ii\varphi_1} & e^{-\ii\varphi_2} \\
		e^{-\ii\varphi_1} & 0 & e^{\ii\varphi_3} \\
		e^{\ii\varphi_2} & e^{-\ii\varphi_3} & 0
	\end{pmatrix}\,,
\label{eq:pf_example}
\end{equation}
where $\varphi_1 = \varphi$, $\varphi_2 = {q}\varphi$, and $\varphi_3 = {q}^2 \varphi$, with $\varphi\in[0,2\pi)$ and ${q} \geq 0$ (see Figure~\ref{fig:three_node}). The network that is associated with the weight matrix $\mathbf{W}$ is a complete directed triangle.

In Figure~\ref{fig:pf_example}, we show the magnitude of the spectral radius and highlight the regions of $({q},\varphi)$ space in which the complex PF property is satisfied.
In this example, the sign-complex spectral radius $\tilde{\rho}(\mathbf{W})$ [see Eq.~\eqref{eq:sign_compl_spect_rad}] equals the spectral radius $\rho(\mathbf{W})$ because $\mathbf{W}$ is Hermitian (and hence normal). 
When $\varphi = 0$, the weight matrix $\mathbf{W}$ satisfies the strong PF property. The corresponding spectral radius and eigenvector are $2$ and $(1,1,1)^\top$, respectively.
\begin{figure}
    \centering
    \includegraphics{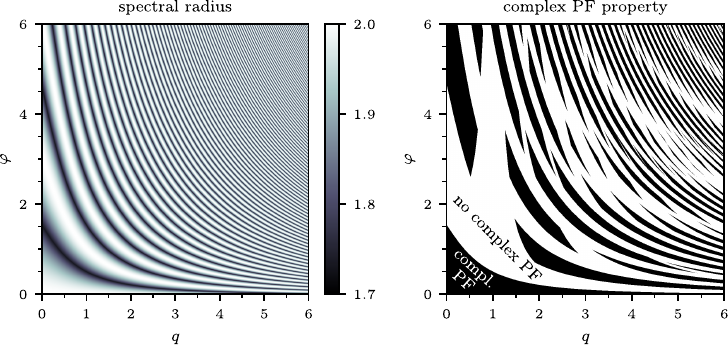}
    \caption{The (a) spectral radius and (b) regions that satisfy the complex PF property for the weight matrix $\mathbf{W}$ [see~\eqref{eq:pf_example}] of the {complete} and directed triangle network with parameters 
    $\varphi_1 = \varphi$, $\varphi_2 = {q}\varphi$, and $\varphi_3 = {q}^2 \varphi$. In (b), the black regions satisfy the complex PF property and the white regions do not.}
    \vspace{-10pt}
    \label{fig:pf_example}
\end{figure}
\subsection{Comparison between Generalizations}
\label{sec:comparison_generalizations}
The generalized eigenvalue problem~\eqref{eq:sign_compl_spect_rad} that was proposed by Rump~\cite{rump2003perron} takes a different perspective from a standard linear eigenvalue problem {$\mathbf W \mathbf x = \lambda \mathbf x$}. In Rump's formulation, the quation $|\mathbf{W} \mathbf{x}| = |\lambda \mathbf{x}|$ requires invariance only of the entrywise magnitudes of the vector $\mathbf{x}$ under the action of the weight matrix $\mathbf{W}$. However, the standard linear eigenvalue problem requires invariance of the complex entries (including both magnitude and phase) of $\mathbf{x}$ because each entry must satisfy $(\mathbf{W} \mathbf{x})_i = \lambda x_i$. Therefore, any eigenvector is also a generalized eigenvector in Eq.~\eqref{eq:sign_compl_spect_rad}, but the converse does not hold.
Additionally, the sign-complex spectral radius is at least as large as the spectral radius. The sign-complex spectral radius {$\tilde{\rho}(\mathbf{W})$} of a normal matrix $\mathbf{W}$ equals its spectral radius {$\rho(\mathbf{W})$}. See Theorem 2.6 in~\cite{rump2003perron}.

In the generalized eigenvector $\mathbf{x}$ that is associated with the sign-complex spectral radius~\cite{rump2003perron}, 
we consider $\abs{\mathbf{x}}\coloneqq(|x_1|,\ldots,|x_N|)^\top$ as a generalization of the PF eigenvector to complex-valued matrices. We refer to this generalized PF eigenvector as the ``Rump PF'' (RPF) vector. Because every complex square matrix has at least one eigenvalue (by the fundamental theorem of algebra) and thus has at least one eigenvector, there is always a corresponding generalized eigenvector and thus an associated RPF vector. 

For the eigenvector $\mathbf{x}$ that is associated with the spectral radius, the choice of a candidate for the complex version of the PF eigenvector depends on the weight matrix $\mathbf{W}$~\cite{noutsos2012perron}.
\begin{itemize}
    \item If $\mathbf{W}$ satisfies the PF property, then $\mathbf{x}$ is real and non-negative; if $\mathbf{W}$ satisfies the strong PF property, then $\mathbf{x}$ is strictly positive. In both cases, one can use $\mathbf{x}$ as the PF vector.  
     \item If $\mathbf{W}$ satisfies the complex PF property but does not satisfy the PF property, then $\Re(\mathbf{x})$ is non-negative. Moreover, if $\mathbf{W}$ satisfies the strong complex PF property, then $\Re(\mathbf{x})$ is strictly positive. In both cases, we take $\Re(\mathbf{x})$ as the PF vector, following the proposal of Noutsos and Varga~\cite{noutsos2012perron}, and we refer to it as the ``Noutsos--Varga PF'' (NVPF) vector.
\end{itemize}

The different generalizations of the PF eigenvector highlight different aspects of networks and their associated weight matrices. The RPF vector uses only the magnitudes of generalized eigenvector entries, whereas the PF eigenvector directly uses the eigenvector $\mathbf{x}$ and the NVPF vector uses the real parts of $\mathbf{x}$. These vectors differ in general, but they coincide in the special case that $\mathbf{W}$ satisfies the PF property. In this case, the PF eigenvector is real and non-negative, so $\Re(\mathbf{x}) = \mathbf{x} = \abs{\mathbf{x}}$. 

When one multiplies an eigenvector by an arbitrary complex scalar, it remains an eigenvector with the same eigenvalue. However, an eigenvector with positive real part does not necessarily still have a positive real part after such a multiplication.
Therefore, if a weight matrix $\mathbf{W}$ satisfies the strong complex PF property with eigenvector $\mathbf{x}$, then the NVPF vector $\Re(\mathbf{x})$ is strictly positive by definition. One can then determine the phase factors $e^{\ii\varphi}$ for which $\Re(e^{\ii\varphi}\mathbf{x})$ remains strictly positive. Analogously, if a weight matrix $\mathbf{W}$ satisfies the complex PF property, one can determine the phase factors $e^{\ii\varphi}$ for which $\Re(e^{\ii\varphi}\mathbf{x})$ remains non-negative. This allows us to characterize a network in terms of the set of complex phase factors that preserve the non-negativity (or positivity) of the NVPF vector. The study of such complex phase factors is relevant in the Noutsos--Varga PF setting~\cite{noutsos2012perron}. It does not apply to Rump's generalization of PF theory~\cite{rump2003perron}, as his definition of the sign-complex spectral radius~\eqref{eq:sign_compl_spect_rad} is invariant under complex scalar multiplication.

A network and its associated weight matrix may not have PF or NVPF vectors. (There is always an RPF vector~\cite{rump2002variational}.) Noutsos and Varga~\cite{noutsos2012perron} showed that the complex PF property holds for non-nilpotent weight matrices $\mathbf{W}$ (\ie, when $\mathbf{W}^k \ne \mathbf{0}$ for any positive integer $k$) if the real part of $\mathbf{W}$ is eventually non-negative. See Theorem 2.3 in \cite{noutsos2012perron}. 
\section{Networks with Generalized PF Properties}
\label{sec:existence_results}
In this section, we study which types of networks satisfy {the generalized PF properties in Table~\ref{tab:generalized_PF}}. To do so, we (1) define the notion of \emph{generalized-switching equivalence} as a method to classify networks based on their eigenvalues and (2) examine ``generalized-switching functions'' to determine specific requirements on eigenvectors.
\subsection{Generalized-Switching Equivalence}
Networks in the same switching-equivalence class have the same eigenvalues, so the notions of switching operations and switching equivalence have been used widely to classify undirected networks~\cite{lange2015magL,zaslavsky1982signed}.
We consider directed networks, so we extend these notions. To do so, we first introduce the notion of a generalized signature, which in turn allows us to define a generalized-switching operation.

\begin{definition}[generalized signature]
    Let $G$ be a directed network, and let $\Gamma$ be a group. A \emph{generalized signature} of $G$ is a map $s\colon E \to \Gamma$. For an edge $e = (i,j)\in E$, we write $s_{ij} \coloneqq s(e) = s((i,j))$.
    We use $s_1$ to denote the trivial generalized signature, which satisfies $s_1(e) \equiv \mathrm{id}$, where $\mathrm{id}$ is the identity element of $\Gamma$. 
    \label{def:signature}
\end{definition}

One can associate the generalized signature of a network with weights of the network's edges. The generalized signature can encode either all of the weight information or only specific information that one obtains from edge weights.
For example, in a network with real edge weights, one can let
$\Gamma = (\mathbb{R}\backslash\{0\},\times)$ and define $s_{ij}$ to be the weight of the edge $e = (i,j)\in E$. Similarly, in a network with complex edge weights, one can take $\Gamma = (\mathbb{C}\backslash\{0\},\times)$ and let $s_{ij}$ be the complex weight of $(i,j) \in E$.
As an example of encoding only partial information in a generalized signature, one can include only the signs of real-valued edges by letting
$\Gamma = (\{1,-1\},\times)$ and defining $s_{ij}$ to be the sign of the edge
$(i,j)\in E$. This last construction coincides the notion of signatures in undirected signed networks \cite{zaslavsky1982signed,Atay2020signedCheeger} (though the edges $(i,j)$ and $(j,i)$ can have different signs in a directed network).

We define the notion of ``induced generalized signature'' for networks with complex edge weights.

\begin{definition}[induced generalized signature]
Let $G$ be a directed complex-weighted network, and let $\Gamma = (\{x\in\mathbb{C} \colon |x| = 1\}, \times)$. The \emph{induced generalized signature} of $G$ is the generalized signature $s$ with entries
\begin{align}
    s_{ij} = e^{\ii\varphi_{ij}} \,\,\,  \text{for all} \,\,\,  e = (i,j)\in E\,,
\end{align}
where $w_{ij} = r_{ij}e^{\ii\varphi_{ij}}$ is the weight of the edge $(i,j)$.
    \label{def:complex-signature}
\end{definition}

Using the notion of a generalized signature, we now define the generalized-switching function and its associated generalized-switching equivalence.

\begin{definition}[generalized-switching function and generalized-switching equivalence]
    Let $G$ be a directed network with generalized signature $s: E\to \Gamma$. For a function $\tau\colon V\to \Gamma$, the generalized signature $s^{\tau}\colon E\to \Gamma$ satisfies
    \begin{equation}
        s^{\tau}(e) = \tau(i)s(e)\tau(j)^{-1} \,\,\, \text{for all}\,\, \, e = (i,j)\in E\,.
    \label{eq:s_tau_e}
    \end{equation}
    We refer to the function $\tau$ as a \emph{generalized-switching function}. The generalized signatures $s$ and $s'$ are \emph{generalized-switching equivalent} if there exists a generalized-switching function $\tau$ such that $s' = s^\tau$, which entails that
    \begin{align}
        s'(e) = \tau(i)s(e)\tau(j)^{-1} \,\,\, \text{for all}\,\,\, e = (i,j)\in E\,.
    \end{align}
    \label{def:switching}
\end{definition}
Generalized switching gives an equivalence relation on the set of generalized signatures with the same edge set. Therefore, one can classify each generalized signature as a member of one switching equivalence class. 

Undirected networks whose signatures belong to the same switching equivalence class have the same eigenvalues~\cite{zaslavsky2013matrices}. Proposition \ref{pro:switch-eigen-invar} guarantees that this invariance property extends to generalized signatures on directed networks.

{\begin{proposition}[eigenvalue invariance]
       Let $G$ be a directed network with edge set $E$ and induced generalized signature $s$, and let $G'$ be a directed network with the same edge set $E$ and (potentially different) induced generalized signature $s'$. If $s$ and $s'$ are generalized-switching equivalent, then their associated weight matrices $\mathbf{W}_{s}$ and $\mathbf{W}_{s'}$, which have the entries
    \begin{align*}
  	(\mathbf{W}_{s})_{ij} &= \begin{cases} s((i,j)) \,, \quad  (i,j)\in E \\ 0\,, \qquad \qquad \textrm{otherwise}\,, \end{cases}  \\
	(\mathbf{W}_{s'})_{ij} &= \begin{cases} {s'}((i,j)) \,, \quad  (i,j)\in E \\ 0\,, \qquad \qquad \textrm{otherwise}\,, \end{cases}
    \end{align*}
    have the same eigenvalues.
    Furthermore, if $\mathbf{x} = (x_i)$
    is an eigenvector of $\mathbf{W}_{s}$, then $\mathbf{x}' = (\tau(i)x_i)$ is an eigenvector of $\mathbf{W}_{s'}$, where $\tau$ is a generalized-switching function that transforms $s$ to $s'$ [see Eq.~\eqref{def:switching}]. 
    \label{pro:switch-eigen-invar}
\end{proposition}}

\begin{proof}
    Because $s$ and $s'$ are generalized-switching equivalent, there is a generalized-switching function $\tau\colon V\to \Gamma$ such that 
    \begin{align*}
         s'(e) = \tau(i)s(e)\tau(j)^{-1} \,\,\,  \text{for all}\,\, e = (i,j)\in E\,.
    \end{align*}
    Therefore,
    \begin{equation}
        \mathbf{W}_{s'} = \mathbf{D}(\tau)\mathbf{W}_s\mathbf{D}(\tau)^{-1}\,,
        \label{eq:w_tau_transform}
    \end{equation}
    where $\mathbf{D}(\tau)$ is the diagonal matrix with entries $(\mathbf{D}(\tau))_{ii} = \tau(i)$.
    An eigenvector $\mathbf{x}$ of $\mathbf{W}_s$ with eigenvalue $\lambda$ satisfies $\mathbf{W}_s\mathbf{x} = \lambda\mathbf{x}$, so  $\mathbf{W}_{s'}(\mathbf{D}(\tau)\mathbf{x}) = \mathbf{D}(\tau)\mathbf{W}_s\mathbf{x} = \lambda\mathbf{D}(\tau)\mathbf{x}$. That is, $\mathbf{D}(\tau)\mathbf{x}$ is an eigenvector of $\mathbf{W}_{s'}$ with the 
    eigenvalue $\lambda$. 
    Therefore, (1) $\mathbf{W}{(\tau)}$ and $\mathbf{W}$ are similar and have the same eigenvalues and (2) their eigenvectors are related via $\mathbf{D}(\tau)$.
\end{proof}

{Proposition~\ref{pro:switch-eigen-invar} applies to any generalized signature on a directed network. As a specific case, we consider complex-weighted networks with induced generalized signatures (see Definition~\ref{def:complex-signature}).
{The weight $w_{ij} = r_{ij}e^{\ii\varphi_{ij}}$ of edge $(i,j)$ in a network $G$ has both a magnitude $r_{ij}$ and a phase $\varphi_{ij}$. Let $\mathbf{W}$ denote the weight matrix whose
nonzero entries are the complex edge weights [\ie,
$(\mathbf{W})_{ij} = w_{ij}$ for each edge $(i,j)$], and recall that the induced
generalized signature encodes only phase information. 
Therefore, in addition to the condition on the induced generalized signature in Proposition~\ref{pro:switch-eigen-invar}, we also require a condition on the
magnitudes. In particular, for the weight matrix $\mathbf{W}$ of $G$ to be similar to the weight matrix $\mathbf{W}'$ of a network $G'$ with the same
edge set and edge weights $w'_{ij} = r'_{ij}e^{\ii\varphi'_{ij}}$, it is necessary (see Corollary~\ref{cor:switch-eigen-invar-complex}) that $r_{ij} = r'_{ij}$ for each edge $(i,j)$.

\begin{corollary}[eigenvalue invariance for induced generalized signatures]
   {Let $G$ be a directed and complex-weighted network with edge set $E$ and induced generalized signature $s$, and let $G'$ be a directed and complex-weighted network with the same edge set $E$ and (potentially different) induced generalized signature $s'$.} Suppose for each {edge} $(i,j)\in E$ that the {weights of the edges in $G$ and $G'$ have the same magnitudes}.
    If $s$ and $s'$ are {generalized-switching} equivalent, then the
    weight matrices $\mathbf{W}$ and $\mathbf{W}'$ have the same eigenvalues.
    Furthermore, if $\mathbf{x} = (x_i)$
    is an eigenvector of $\mathbf{W}$, then $\mathbf{x}' = (\tau(i)x_i)$ is an eigenvector of $\mathbf{W}'$, where $\tau$ is a {generalized-switching} function that transforms $s$ to $s'$ [see Eq.~\eqref{def:switching}]. 
    \label{cor:switch-eigen-invar-complex}
\end{corollary}

We prove Corollary~\ref{cor:switch-eigen-invar-complex} in the Supplemental Material.

\medskip

For PF properties, it is relevant to consider the eigenvector(s) that are associated with the dominant eigenvalues. 
Therefore, in addition to specifying a generalized-switching equivalence class, we also need to consider which generalized-switching functions yield the desired eigenvector properties. 
In Propositions \ref{pro:swich-lam1-no} and \ref{pro:switch-lam1-yes}, we demonstrate how generalized-switching equivalence classes help identify PF properties.

\begin{proposition}\label{pro:swich-lam1-no}
    If a weight matrix $\mathbf{W}$ does not have a dominant eigenvalue $\lambda_1 > 0$, then none of the weight matrices in
    the same generalized-switching equivalence class satisfy the {PF property, the strong PF property, the complex PF property, or the strong complex PF property}.
\end{proposition}

\begin{proposition}
    If a weight matrix $\mathbf{W}$ has a dominant eigenvalue $\lambda_1 > 0$, then there exists a generalized-switching function $\tau$ such that ${\mathbf{W}'} = \mathbf{D}(\tau)\mathbf{W}\mathbf{D}(\tau)^{-1}$ satisfies the PF property and the complex PF property. Furthermore, if the dominant eigenvalue is simple and its associated eigenvector does not have any $0$ entries, then there exists a generalized-switching function $\tau$ such that ${\mathbf{W}'} = \mathbf{D}(\tau)\mathbf{W}\mathbf{D}(\tau)^{-1}$ satisfies the strong PF property and the strong complex PF property.
    \label{pro:switch-lam1-yes}
\end{proposition}

We prove Propositions~\ref{pro:swich-lam1-no} and \ref{pro:switch-lam1-yes} in the Supplemental Material.

\medskip

In the traditional PF theorem, the eigenvector that is associated with the dominant eigenvalue has all positive entries for a non-negative matrix. That eigenvector is unique up to multiplication by a real scalar.
This property does not carry over to weight matrices with complex entries.
However, if the weight matrix $\mathbf{W}$ of a complex-weighted network is normal, then its eigenvectors are orthogonal to each other. Consequently, the dominant eigenvector of $\mathbf{W}$ is unique up to multiplication by a scalar if $\mathbf{W}$ is normal, satisfies one of the generalized PF properties in Table~\ref{tab:generalized_PF}, and has a generalized PF eigenvector with all positive entries. 

To obtain a unique complex PF eigenvector or a unique strong complex PF eigenvector from the generalized-switching operation, the entries of the eigenvectors of ${\bf W}$ must differ in their phases. Suppose instead that there is some other eigenvector $\mathbf{y} \neq \mathbf{x}$ that has the same entrywise phases as  $\mathbf{x}$, with zeros occurring in the same entries. 
For any generalized-switching function $\tau$ with $\Re(\tau(i)x_i) \ge 0$ for all $i\in \{1,\ldots,N\}$, it is then necessarily true that
$\Re(\tau(i)y_i)\ge 0$ for all $i \in \{1,\ldots,N\}$. 
Therefore, there is no generalized-switching function $\tau$ for which $\mathbf{D}(\tau)\mathbf{x}$ is the unique (up to multiplication by a scalar) dominant eigenvector whose entries have non-negative real parts.
Additionally, there is no unique (up to multiplication by a scalar) strong complex PF eigenvector that satisfies $\Re(\tau(i)x_i) > 0$ for all $i \in \{1,\ldots,N\}$.

By Proposition~\ref{pro:switch-lam1-yes}, {if $\mathbf{W}$} has a dominant eigenvalue $\lambda_1 > 0$, there is necessarily a set of complex-weighted networks that satisfy the PF property and the complex PF property. For example, $\mathbf{W}$ satisfies the complex PF property if all of its edge weights are positive. By appropriately choosing a generalized-switching function $\tau$, one can then obtain a set of complex-weighted networks that satisfy the complex PF property. The corresponding 
generalized-switching equivalence class is {closely related to} the traditional structural-balance class~\cite{zaslavsky1982signed,Easley_2010_networks,newman2018networks} (see Section~\ref{sec:structural_balance_class}). Furthermore, if the dominant eigenvalue is simple and its associated eigenvector does not have any $0$ entries, then Proposition~\ref{pro:switch-lam1-yes} also guarantees that there is a generalized-switching function $\tau$ such that a network with the complex-valued weight matrix 
${\mathbf{W}'}$ satisfies both the strong PF property and the strong complex PF property. If all of the edge weights are negative and the associated network is not bipartite, then the dominant eigenvalue of $\mathbf{W}$ is negative. By Proposition~\ref{pro:swich-lam1-no}, it follows that no weight matrix in the same generalized-switching equivalence class as $\mathbf{W}$ satisfies the PF property, the strong PF property, the complex PF property, or the strong complex PF property.
\subsection{Example: Generalized Structural-Balance Class} 
\label{sec:structural_balance_class}
In this subsection, we discuss a specific generalized-switching equivalence class that is closely related to the traditional structural-balance class.\footnote{The traditional \emph{structural-balance class}, which more specifically is a ``strong-structural-balance class'', has been studied extensively for signed networks, where edge weights are either positive or negative. Signed networks arise in many applications in international relations, sociology, economics, biology, and other fields~\cite{diazdiaz2025}. A key notion in the study of signed networks is strong ``structural balance''~\cite{Easley_2010_networks,Facchetti_2011_large,Kunegis_2009_Zoo,SzellEtal_2010_multirelational}, which requires that all cycles in a network have an even number of negative edges. 
Building on ideas of structural balance, researchers have studied the spectral properties of signed networks~\cite{Atay2020signedCheeger,Kunegis2010signspect,zaslavsky1982signed} and have examined their implications for system dynamics~\cite{altafini2013consensus,tian2024sign}. Ideas of structural balance have also been relaxed for studies of community structure in signed networks \cite{traag2009,macon2012}, which in turn influenced the development of some methods to detect community structure in multilayer networks \cite{mucha2010}.} The structural-balance class has been examined previously for Hermitian weight matrices (see, \eg, \cite{lange2015magL,tian2023complex}). We relax this assumption and propose a notion of generalized structural balance using the induced generalized signature in Definition \ref{def:complex-signature}.

\begin{definition}[generalized structural balance]
    Let $G$ be a directed and complex-weighted network with induced generalized signature $s$. The network $G$ satisfies \emph{generalized structural balance} if $s$ is generalized-switching equivalent to the trivial generalized signature $s_1$ (which satisfies $s_1(e) = \mathrm{id} = 1$ for all $e\in E$).
    \label{def:gen-balance}
\end{definition}

Consider a cycle
\begin{equation}
    (j_1,j_2), (j_2,j_3), \ldots, (j_{l-1},j_l), (j_l,j_1)
\end{equation}
in a directed and weighted network $G$ {with induced generalized signature $s$ on the edge set}. The value of the induced generalized signature of this cycle is the conjugacy class in $\Gamma$ of the group element
\begin{equation}
    s_{j_1j_2}\, s_{j_2j_3}\cdots s_{j_{l-1}j_l}\, s_{j_lj_1}\,,
    \label{eq:cyc-conjug}
\end{equation}
where we use the same notation $s_{ij} = s(e) = s((i,j))$ as in Definition \ref{def:signature}.

In Proposition \ref{pro:bal-cycle-1}, we show that the value of the induced generalized signature of every cycle is $1$ in a generalized-structurally-balanced network.

\begin{proposition}
    In a directed and complex-weighted network $G$ satisfies generalized structural balance, the value of the induced generalized signature of every cycle in $G$ is $1$.
    \label{pro:bal-cycle-1}
\end{proposition}

\begin{proof}
    Because $G$ satisfies generalized structural balance, the induced generalized signature $s$ is generalized-switching equivalent to the trivial generalized signature $s_1$. Therefore, there is a generalized-switching function $\tau\colon V\to \Gamma$ such that 
     \begin{align*}
         s(e) = \tau(i)s_1(e)\tau(j)^{-1} = \tau(i)\tau(j)^{-1} \,\,\, \text{for all}\,\,\, e = (i,j)\in E\,.
    \end{align*}
    Consequently, the induced generalized signature value of each cycle is
    \begin{align*}
          &s_{j_1j_2}\, s_{j_2j_3}\cdots s_{j_{l - 1}j_l}\, s_{j_lj_1} \\
        &\qquad = \tau(j_1)\tau(j_2)^{-1}\tau(j_2)\tau(j_3)^{-1}\cdots \tau(j_{l-1})\tau(j_l)^{-1}\tau(j_l)\tau(j_1)^{-1} = 1\, .
    \end{align*}
\end{proof}

For networks with complex weights, we define the ``phase'' of a cycle to be the sum (modulo $2\pi$) of the phases of its edges. 
In a complex-weighted network that satisfies generalized structural balance, all cycles have phase $0$.

The weight matrix of a complex-weighted network that satisfies generalized structural balance is similar to a non-negative matrix.
 
\begin{proposition}
    Let $G$ be a directed and complex-weighted network. If $G$ satisifes generalized structural balance, its weight matrix $\mathbf{W}$ is similar to the non-negative matrix $|\mathbf{W}|$ (where we recall that we take absolute values entrywise).
    \label{pro:gen-bal-sim}
\end{proposition}

\begin{proof}
     Because $G$ satisfies generalized structural balance, the induced generalized signature $s$ is generalized-switching equivalent to the trivial generalized signature $s_1$. Therefore, there is a generalized-switching function 
     $\tau\colon V\to \Gamma$ such that 
     \begin{align*}
         s(e) = \tau(i)s_1(e)\tau(j)^{-1} = \tau(i)\tau(j)^{-1} \,\,\, \text{for all}\,\,\,  e = (i,j)\in E\,.
    \end{align*}
    Consequently, {as in} the proof of Corollary \ref{cor:switch-eigen-invar-complex}, we have 
    \begin{align*}
        w_{ij} = r_{ij}s(e) = r_{ij}\tau(i)\tau(j)^{-1} = \tau(i)|w_{ij}|\tau(j)^{-1}\, ,
    \end{align*}
    where $w_{ij} = r_{ij}e^{\ii \varphi_{ij}}$ is the complex weight of the edge $e = (i,j)\in E$. Therefore, 
     \begin{align*}
         \mathbf{W} = \mathbf{D}(\tau)|\mathbf{W}|\mathbf{D}(\tau)^{-1}\,,
     \end{align*}
     where $\mathbf{D}(\tau)$ is the diagonal matrix with entries $(\mathbf{D}(\tau))_{ii} = \tau(i)$.
\end{proof} 

Let $\mathbf{W}$ be the weight matrix of a generalized-structurally-balanced complex-weighted network $G$. 
By Proposition \ref{pro:gen-bal-sim}, $\mathbf{W}$ and $|\mathbf{W}|$ are similar and thus have identical spectra. If $\mathbf{x}$ is the
PF eigenvector of $|\mathbf{W}|$, it follows that $\mathbf{D}(\tau)\mathbf{x}$ is {a} dominant eigenvector of $\mathbf{W}$.

We now characterize which generalized-switching functions transform the generalized signature of $|\mathbf{W}|$ (which necessarily has a trivial generalized signature) to obtain a weight matrix with an induced generalized signature that preserves the PF, strong PF, complex PF, and strong complex PF properties through how they change the PF eigenvector. 
\begin{itemize}

\item[(1)]\label{item:PF_strong_PF}\emph{PF and strong PF properties.} The PF property requires a dominant eigenvector to have non-negative real entries, whereas the strong PF property requires strict positivity. 
When $|\mathbf{W}|$ is irreducible, its PF eigenvector $\mathbf{x}$ is strictly positive. 
In this case, the vector $\mathbf{D}(\tau)\mathbf{x}$ is real if and only if $\tau(i) = 1$ for all nodes $i$. 
Therefore, only the trivial generalized-switching function preserves the PF property, and this is also true for the strong PF property.

\medskip

\item[(2)]{\emph{Complex PF and strong complex PF properties.}
The complex PF property requires the real parts of a dominant eigenvector to be non-negative, and the strong complex PF property requires strict positivity of the real parts of a dominant eigenvector. The vector $\mathbf{D}(\tau)\mathbf{x}$ is a dominant eigenvector, and $e^{\ii \phi}\mathbf{D}(\tau)\mathbf{x}$ is a dominant eigenvector for any phase $\phi$. Therefore, to satisfy the complex PF property, there necessarily exists a phase $\phi$ such that
\begin{equation}
    \Re\bigl({e^{\ii \phi}} \tau(i)\bigr) \ge 0 \,\,\, \text{for all}\,\,\,  i\, . 
\label{equ:CPF:balance}
\end{equation}

Therefore, any generalized-switching function $\tau$ that satisfies~\eqref{equ:CPF:balance} yields a matrix ${\mathbf{W}}$ whose dominant eigenvector satisfies the complex PF property. For the strong complex PF property, we replace the inequality in~\eqref{equ:CPF:balance} with a strict inequality.
}
\end{itemize}

For a Hermitian weight matrix $\mathbf{W}$, generalized structural balance reduces to traditional structural balance. It was proven in \cite{tian2023complex} that one can define the structural-balance class in complex-weighted networks using a partition $\{V_1, V_2, \ldots, V_k\}$ of the set of nodes. 
We refer to each unit $V_h\subseteq V$ as a ``block'' of a partition. In a partition, recall that $V_h\cap V_l = \emptyset$ if $h\ne l$ and $\cup_{h = 1}^kV_h = V$.
A complex-weighted network is structurally balanced if (1) every edge whose attached nodes are in the same block $V_h$ as each other has phase $0$, (2) all edges from nodes in block $V_h$ to nodes in block $V_l \neq V_h$ have the same phase $\theta_{hl}$, and (3) when one considers each block as a single ``supernode'', the sum of the phases along the edges of any cycle of supernodes is a multiple of $2\pi$. {We refer to a partition that satisfies properties (1)--(3) as a ``characteristic partition''.}

Let $\{V_1, V_2, \ldots, V_k\}$ be a characteristic partition of the set of nodes of a structurally balanced complex-weighted network with weight matrix $\mathbf{W}$.
All paths that connect nodes in $V_{\sigma(i)}$ to nodes in $V_{\sigma(j)}$, where $\sigma\colon V\to \{1,\ldots, k\}$ assigns each node to a block of the partition, {must} have the same phase.
Therefore, we consider a ``characteristic generalized-switching function'' $\tau$ with entries
\begin{equation}
    \tau(i) = \exp(\ii \theta(P_{1i})) = \exp(\ii \theta(P_{\sigma(1)\sigma(i)}))\,,
\end{equation}
where $P_{1i}$ is a path from node $1$ to node $i$, the path $P_{\sigma(1)\sigma(i)}$ starts from a node in block $V_{\sigma(1)}$ and terminates at a node in block $V_{\sigma(i)}$, and $\theta(\cdot)$ returns the sum (modulo $2\pi$) of the phases that are associated with the edges of the path.

When one applies the characteristic generalized-switching function $\tau$ to a network with the trivial generalized signature $s_1 = \mathrm{id}$, one obtains $\mathbf{D}(\tau) |\mathbf{W}| \mathbf{D}(\tau)^{-1} = \mathbf{W}$. According to the condition~\eqref{equ:CPF:balance}, {the generalized-switching-function entries must satisfy $\Re (e^{\ii \phi} \tau(i)) \ge 0$ for all $i$ and some phase $\phi$.
Given the parameterization $\tau(i) = \exp(\ii \theta(P_{\sigma(1)\sigma(i)}))$, there is a phase $\phi$ such that
\begin{equation}\label{eq:balance-cond-path}
     \Re (e^{\ii \phi} e^{\ii \theta(P_{\sigma(1)\sigma(i)})}) = \cos\bigl(\phi+\theta(P_{\sigma(1)\sigma(i)}) \bigr) \ge 0 \,\,\, \text{for all}\,\,\, i\in V\,.
\end{equation}
That is, the phases of all paths that start at node 1 must lie (modulo $2\pi$) in a closed interval of length $\pi$.
\begin{figure}
    \centering
    \begin{tikzpicture}[>=stealth, thick]

\node[anchor=north west] at (-2.5,2.5) {(a)};

\node[circle, draw, fill=gray!20, minimum size=15pt] (1) at (90:2)  {1};
\node[circle, draw, fill=gray!20, minimum size=15pt] (2) at (210:2) {2};
\node[circle, draw, fill=gray!20, minimum size=15pt] (3) at (330:2){3};

\draw[->, bend left=20, draw=taborange!100]
  (1) to node[pos=0.65, right, taborange!100] {\( \theta_1+\theta_2 \)} (2);

\draw[->, bend left=20, draw=tabgreen!100, dashed]
  (2) to node[midway, below] {} (3);

\draw[->, bend left=20, draw=tabblue!100, dashed]
  (3) to node[midway, right] {} (1);

\draw[->, bend left=20, draw=taborange!100, dashed]
  (2) to node[midway, left] {} (1);

\draw[->, bend left=20, draw=tabgreen!100]
  (3) to node[midway, below, tabgreen!100] {\( \theta_2 \)} (2);

\draw[->, bend left=20, draw=tabblue!100]
  (1) to node[midway, right, tabblue!100] {\( \theta_1 \)} (3);

\end{tikzpicture}
    \hspace{5em}
    \begin{tikzpicture}[>=stealth, thick]

\node[anchor=north west] at (-2.5,2.5) {(b)};
    
\node[circle, draw, fill=gray!20, minimum size=25pt] (1) at (90:2)  {};

\filldraw[gray] (-0.25,2) circle (0.1);
\filldraw[gray] (0.25,2) circle (0.1);

\draw[thin, <->] (-0.15,2) -- (0.15,2);

\node[circle, draw, fill=gray!20, minimum size=25pt] (2) at (210:2) {};

\filldraw[gray] (-2*0.866025-0.25,-1) circle (0.1);
\filldraw[gray] (-2*0.866025+0.25,-1) circle (0.1);

\draw[thin, <->] (-2*0.866025-0.15,-1) -- (-2*0.866025+0.15,-1);

\node[circle, draw, fill=gray!20, minimum size=25pt] (3) at (330:2){};

\filldraw[gray] (2*0.866025-0.25,-1) circle (0.1);
\filldraw[gray] (2*0.866025+0.25,-1) circle (0.1);

\draw[thin, <->] (2*0.866025-0.15,-1) -- (2*0.866025+0.15,-1);

\draw[->, bend left=20, draw=taborange!100]
  (1) to node[pos=0.65, right, taborange!100] {\( \theta_1+\theta_2 \)} (2);

\draw[->, bend left=20, draw=tabgreen!100, dashed]
  (2) to node[midway, below] {} (3);

\draw[->, bend left=20, draw=tabblue!100, dashed]
  (3) to node[midway, right] {} (1);

\draw[->, bend left=20, draw=taborange!100, dashed]
  (2) to node[midway, left] {} (1);

\draw[->, bend left=20, draw=tabgreen!100]
  (3) to node[midway, below, tabgreen!100] {\( \theta_2 \)} (2);

\draw[->, bend left=20, draw=tabblue!100]
  (1) to node[midway, right, tabblue!100] {\( \theta_1 \)} (3);
  
\end{tikzpicture}
    \caption{Examples of structurally-balanced complex-weighted networks with (a) $3$ nodes and (b) $3$ blocks, where edges of the same color but different textures have opposite phases.
    In both examples, the condition \eqref{equ:CPF:balance} for the associated {Hermitian} weight matrix ${\bf W}$ to satisfy the complex PF property holds when \eqref{eq:balance-cond-path} is satisfied.
    }        
    \label{fig:bal-3}
\end{figure}
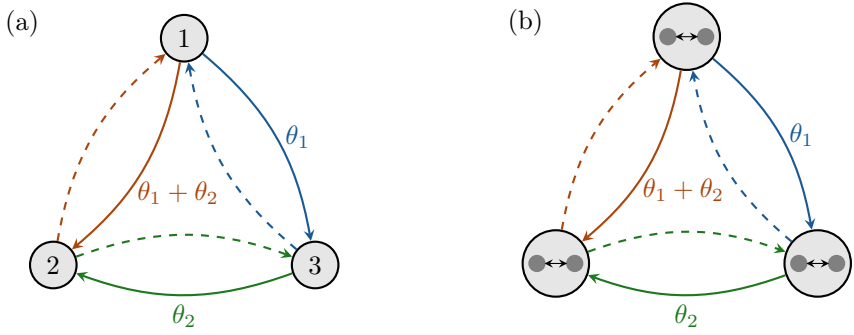

As examples, we show a structurally-balanced complex-weighted network with $3$ nodes in Figure~\ref{fig:bal-3}(a) and a structurally-balanced complex-weighted network with $3$ blocks in Figure~\ref{fig:bal-3}(b). In both examples, the weight matrices are Hermitian. The complex PF property requires that all paths that start at node $1$ have sufficiently close phases. Specifically, their phases must be contained in an interval of length $\pi$, as specified in \eqref{eq:balance-cond-path}. We obtain sufficient (but not necessary) conditions by specifying a value of $\phi$ in \eqref{eq:balance-cond-path}. For instance, with $\phi = 0$, the two example networks in Figure~\ref{fig:bal-3} satisfy the complex PF property if the phases satisfy the following conditions.
For $\theta_1,\theta_2\in[0,2\pi)$, the complex PF property holds when (1) $\theta_1 \le \pi/2$ and $\theta_2 \ge 3\pi/2$ (\ie, $\theta_1$ is small and $\theta_2$ is large) or when (2) $\theta_1 \ge 3\pi/2$ and $\theta_2 \le \pi/2$ (\ie, $\theta_2$ is small and $\theta_1$ is large).
When $\theta_1$ and $\theta_2$ are both small, we require that (3a) $\theta_1 + \theta_2 \le \pi/2$; when they are both large, we require that (3b) $\theta_1 + \theta_2 \ge 7\pi/2$. We obtain conditions for the strong complex PF property by replacing the non-strict inequalities with strict inequalities.

The results in this subsection also provide some insight into Figure~\ref{fig:pf_example}.
When $\varphi = 0$, the edge weights are $1$ and the weight matrix ${\bf W}$ satisfies the strong PF property, which implies that it satisfies the complex PF property for all values of ${q}$. 
The parameter $q$ determines the scalings of the phases $\varphi_2 = q\varphi$ and $\varphi_3 = q^2\varphi$ in
Eq.~\eqref{eq:pf_example}.
In Figure~\ref{fig:pf_example}, we observe a recurring pattern for the complex PF property as we vary ${q}$. As we increase ${q}$, the phases ${q}\varphi$ and ${q}^2\varphi$ wind more rapidly modulo $2\pi$, leading to increasingly frequent alternations between parameter values for which the complex PF property holds and fails. For ${q} = 0$ and $\varphi > 0$, similarly to condition~\eqref{eq:balance-cond-path}, the complex PF property holds when either $\varphi \lessapprox \pi/2$ or 
$\varphi \gtrapprox 3\pi/2$.

As a final example, consider the 2-node network with Hermitian weight matrix
\begin{equation} \label{eq:2x2_hermitian_weight}
    \mathbf{W} =
		\begin{pmatrix}
			a & b e^{\ii\varphi} \\
			b e^{-\ii\varphi} & c
			\end{pmatrix}\,,
\end{equation}
where $a, b, c \in \mathbb{R}_{\geq 0}$ and $\varphi\in[0,2\pi)$. This network, which was examined by B\"ottcher and Porter~\cite{bottcher2024complex}, is in the structural-balance class.\footnote{Aside from self-edges with phase $0$, there is only one cycle between nodes 1 and 2. The edges of the cycle have opposite phases, so these phases sum to $0$.} Consequently, its weight matrix satisfies the strong PF property when $\varphi=0$ (as discussed in~\cite{bottcher2024complex}), as this is the trivial-switching case. See the discussion on PF and strong PF properties on page~\pageref{item:PF_strong_PF}.
\section{Generalization of Eigenvector Centrality and Related Centralities}
\label{sec:generalized_centralities}
Eigenvector centrality is a popular measure of node importance in  networks~\cite{newman2018networks,central2025}. 
When a network has real and positive edge weights, edges to nodes with larger eigenvector centralities contribute more to a node's eigenvector centrality than edges to nodes with smaller eigenvector centralities. Let $x_i$ denote the eigenvector centrality of node $i$, let $\mathbf{x}$ denote the vector of eigenvector centralities, and let $\mathbf{W}$ denote a network's weight matrix. We then have
\begin{align} \label{eigen}
    \mathbf{W}\mathbf{x} = \lambda \mathbf{x}\,.
\end{align}
where $\lambda$ is a real positive constant.\footnote{Eigenvector centrality has known drawbacks, and researchers have proposed various strategies to mitigate them~\cite{snyder2024primer}. For instance, eigenvector centrality has regimes with localization, which results in biased concentration of the support of the eigenvector that is associated with the dominant eigenvalue of a weight matrix ${\bf W}$ on a small subset of the associated network's nodes~\cite{martin2014local}. One can drastically mitigate localization by using a non-backtracking matrix instead of an adjacency matrix, and one can define non-backtracking matrices for networks with complex weights.} For strongly connected networks, the PF theorem guarantees that $\mathbf{W}$ has a dominant eigenvalue $\lambda_1 = \rho(\mathbf{W})$ with an associated eigenvector $\mathbf{x}$ with {positive} entries.
\subsection{Generalization of Eigenvector Centrality}
\label{sec:generalization_ev_centrality}
For networks with complex weights, we extend the characteristic space of each edge from positive real numbers to the complex plane (excluding the origin). Correspondingly, we also extend the {eigenvector centrality} of each node to the complex plane. That is, the {eigenvector centrality} of node $i$ now takes a scalar value $x_i\in\mathbb{C}\setminus \{0\}$. 
To determine the {eigenvector centrality} along a certain direction in $\mathbb{C}$, one projects the eigenvector $\mathbf{x}$ onto the unit vector in that direction. To understand how the {eigenvector centralities} of nodes depend on their neighbors, one needs to know how the complex weights characterize the flow of {centrality}. Accordingly, we suppose that the edge weight $w_{ij} = r_{ij}e^{\ii\varphi_{ij}}$ encodes the relative {centrality} of node $i$ to node $j$, where we characterize the flow of {centrality} both by the magnitude $r_{ij}$ and the phase $\varphi_{ij}$. As in Eq.~\eqref{eigen}, we then determine the eigenvector centralities of the nodes using the eigenvalue equation
\begin{align}
    \mathbf{W}\mathbf{x} = \lambda \mathbf{x}\,,
    \label{eq:general_ev}
\end{align}
where $\lambda$ is a real positive constant.
When $\mathbf{W}$ only has real non-negative entries, the generalized eigenvector centralities $x_i$ from Eq.~\eqref{eq:general_ev} reduce to the ordinary eigenvector centralities $x_i$ from Eq.~\eqref{eigen}.
{Moreover, when $\mathbf W$ is Hermitian, the dominant eigenvalue and its associated eigenvector yields the best rank-$1$ approximation of $\mathbf W$ in both the Frobenius norm and the matrix $2$-norm (by the Eckart--Young theorem~\cite{eckart1936approximation}). When the dominant eigenvalue is simple, this best rank-$1$ approximation is unique up to multiplication of $\mathbf x$ by a nonzero complex scalar.}

We now examine the directional information that is encoded in our generalized eigenvector centrality. For an eigenvector $\mathbf{x}$ that is associated with the eigenvalue $\lambda$, any vector $e^{i\varphi}\mathbf{x}$ (with $\varphi \in [0,2\pi)$) that we obtain by a global phase rotation is also an eigenvector that is associated with $\lambda$.\footnote{Although it is not formally equivalent, this idea of rotating eigenvectors to obtain interpretable centrality scores bears some resemblance to Wick rotation~\cite{wick1954}, which one uses to transform problems from Minkowski space to Euclidean space to facilitate their solution.} We can thus consider generalized eigenvector centrality along the positive real axis if the phases that are associated with the eigenvector entries are in an interval of length $\pi$. One can then multiply $\mathbf{x}$ by a global phase factor $e^{\ii\varphi}$ to ensure that $\mathrm{Re}\left(e^{\ii\varphi}\mathbf{x}\right) \ge 0$.
In principle, multiple global phase factors can yield $\mathrm{Re}\left(e^{\ii\varphi}\mathbf{x}\right) \ge 0$. However, as in the traditional PF theorem for non-negative matrices, we seek an eigenvector that is associated with the largest eigenvalue $\lambda_1 = \rho(\mathbf{W})$. When $\lambda_1$ is simple, the associated eigenvector is unique up to multiplication
by a nonzero complex scalar. Accordingly, $\mathbf{x}$ is a complex right PF eigenvector (see Table~\ref{tab:generalized_PF}) and its real part gives an eigenvector centrality
\begin{align}     \label{eq:eigen-cent-re}
    c_i = \mathrm{Re}\left(x_i\right)\,.
\end{align}

More generally, one can consider a range of directions $\mathbf{u}$ that yield non-negative real parts and assign a weight $p(\mathbf{u}) > 0$ to each such direction. The eigenvector centrality of node $i$ is then
\begin{align}
    c_i = \int_{\mathbf{u}\in S}p(\mathbf{u})\mathrm{Re}(e^{-\ii\varphi(\mathbf{u})}x_i) \, \mathrm{d}\mathbf{u}\,,
    \label{eq:eigen-cent-set}
\end{align}
where $S$ is the set of directions, $\int_{\mathbf{u}\in S}p(\mathbf{u})\mathrm{d}\mathbf{u} = 1$, and $\varphi(\mathbf{u})$ is the angle of $\mathbf{u}$ with respect to the positive real axis.
In an ideal situation in which $\mathbf{x}$ is also a right PF eigenvector in which each entry is real and non-negative, we obtain the same eigenvector centrality along all directions.
\subsection{Generalization of Other Eigenvector-Based Centralities and Related Centralities}
One can also generalize other eigenvector-based centrality measures to networks with complex weights. As examples, {we consider hubs and authorities~\cite{kleinberg1999authoritative} and PageRank~\cite{gleich2015pagerank}.}
We also consider Katz centrality~\cite{katz1953new}, which depends on an eigenvector but is not itself an eigenvector of a matrix. 

For networks with positive real weights, hub and authority centralities separate contributions from incoming edges (\ie, ``in-edges'') and outgoing edges (\ie, ``out-edges''), so each node has both a hub centrality $\mathbf{c}_{\rm H}$ and an authority centrality $\mathbf{c}_{\rm A}$~\cite{kleinberg1999authoritative}. Intuitively, a node is a better hub if it points to more authorities, and a node is a better authority if more hubs point to it. We express this intuition mathematically using the equations
\begin{align}
        \mathbf{c}_{\rm H} = \gamma_1 \mathbf{W}\mathbf{c}_{\rm A} \,, \,
        \mathbf{c}_{\rm A} = \gamma_2 \mathbf{W}^\top\mathbf{c}_{\rm H}
    \quad \Longleftrightarrow \quad
        \mathbf{c}_{\rm H} = \gamma \mathbf{W}\mathbf{W}^\top\mathbf{c}_{\rm H} \,, \,
        \mathbf{c}_{\rm A} = \gamma\mathbf{W}^\top\mathbf{W}\mathbf{c}_{\rm A} \,,
        \label{eq:gen_hubs_auth}
\end{align}
where $\gamma_1, \gamma_2 > 0$ are real constants and $\gamma = \gamma_1\gamma_2$. 
That is, $\mathbf{c}_{\rm H}$ and $\mathbf{c}_{\rm A}$ are the eigenvectors of $\mathbf{W}\mathbf{W}^\top$ and $\mathbf{W}^\top\mathbf{W}$, respectively, with the same eigenvalue $1/\gamma$.
For a directed network that is strongly connected and has a weight matrix $\mathbf W$ {with} non-negative entries, the {traditional} PF theorem guarantees that both $\mathbf W\mathbf W^\top$ and $\mathbf W^\top\mathbf W$ have a simple eigenvalue that equals their spectral radius and also that their corresponding eigenvectors $\mathbf c_{\rm H}$ and $\mathbf c_{\rm A}$ are strictly positive. When a network's edge weights take complex values, the quantities $\mathbf{c}_{\rm H}, \mathbf{c}_{\rm A} \in \mathbb{C}^N$ give hub and authority values in the complex plane, and one can then consider centralities along a specific direction [as in Eq.~\eqref{eq:eigen-cent-re}] or along a set of directions [as in Eq.~\eqref{eq:eigen-cent-set}]. 

The PageRank vector $\mathbf c_{\rm P}$ is the stationary state of a random walk with teleportation~\cite{gleich2015pagerank}. For a complex-weighted 
network, we extend PageRank centrality by considering random walks with complex-valued transition matrices~\cite{boettcher2021classical,tian2023complex,bottcher2024complex}. We define the PageRank vector $\mathbf{c}_{\rm P}\in\mathbb{C}^M$ as the solution of
\begin{equation}
	\mathbf c_{\rm P} = a \mathbf P^\top \mathbf c_{\rm P} + \frac{1 - a}{N}\mathbf 1 \,,
\label{eq:pagerank}
\end{equation}
where $a\in[0,1/\lambda_1(\mathbf{P}))$ is the teleportation parameter, 
$\mathbf{1}$ is the vector with $1$ in each entry, 
and the transition matrix $\mathbf P$ has entries $P_{ij} = w_{ij}\exp(\ii\varphi_{ij})/d_i$. When $\varphi_{ij} = 0$ for all $(i,j)$, this notion of PageRank reduces to the classical notion of PageRank.
 
We also generalize Katz centrality to networks with complex weight matrices $\mathbf{W}$. The Katz-centrality vector $\mathbf{c}_{\rm K}\in \mathbb{C}^N$ quantifies node importances by aggregating contributions from all possible walks to each node~\cite{newman2018networks}. That is,
\begin{align}
	\mathbf{c}_{\rm K} = \sum_{k = 1}^{\infty} \alpha^k (\mathbf{W}^\top)^k \mathbf{1}\,,
    \label{eq:gen_katz}
\end{align}
where the real parameter $\alpha \in [0,1)$, with $\alpha < 1/\rho(\mathbf{A})$, is a downweighting factor that progressively diminishes the contributions of longer walks.
\section{Examples with Empirical Data}
\label{sec:applications}
In this section, we calculate generalized eigenvector centrality \eqref{eq:eigen-cent-re} 
on complex-weighted networks from quantum physics, circuit {theory}, mathematical chemistry, and social and communication networks (see Figure~\ref{fig:examples}). 
We also compare calculations of generalized eigenvector centrality with calculations of generalized PageRank, generalized Katz centrality, and generalized hubs and authorities.

\begin{figure}
    \centering
    \includegraphics[width=\textwidth]{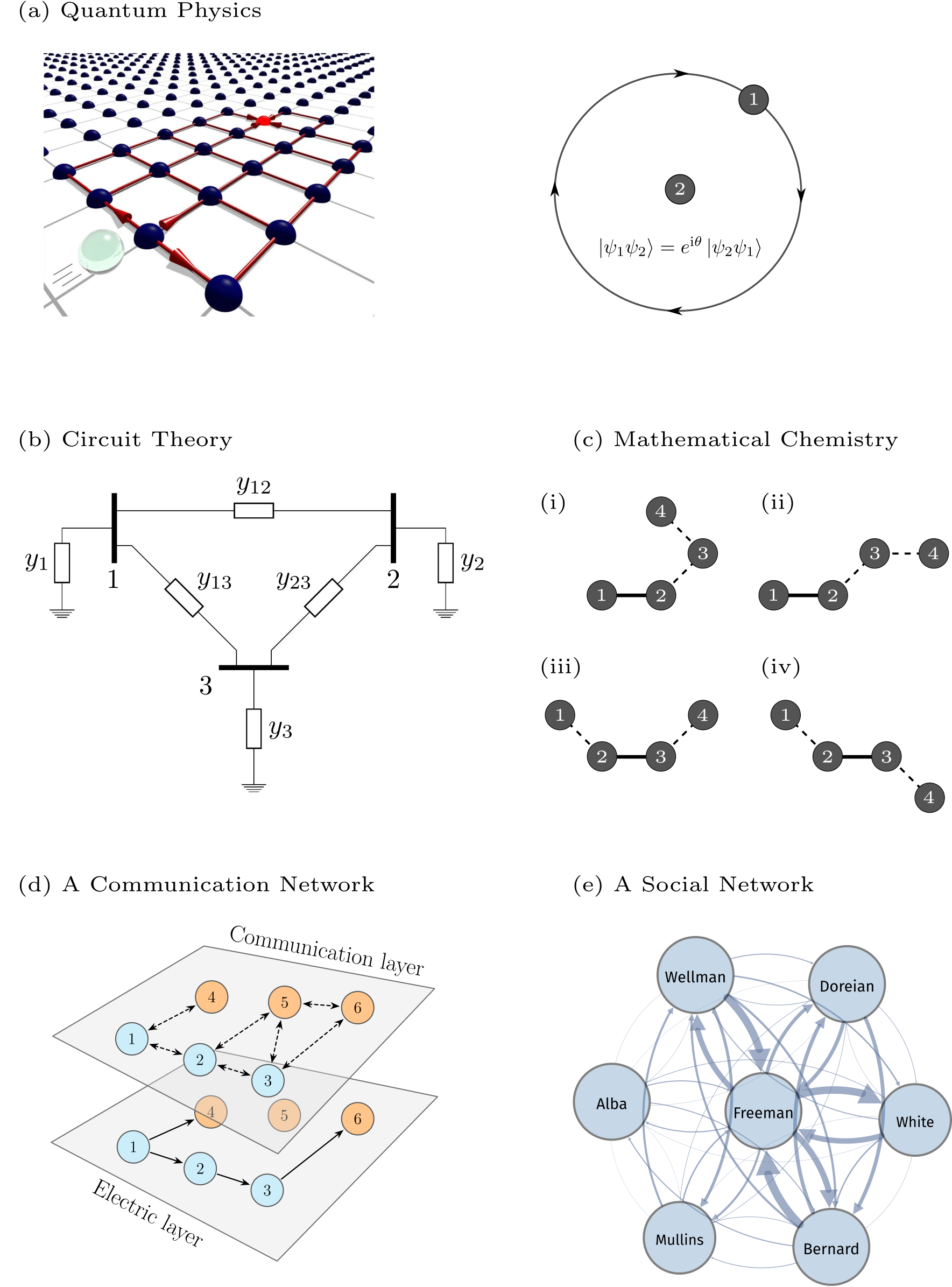}
    \caption{Examples of complex-weighted networks in various applications. (a) In quantum physics, complex weights arise in scattering processes and in the propagators of quantum walks. 
    One can also associate complex phases with permuting anyons.
     (b) An example of a 3-node circuit with 
     complex-valued admittances.
    (c) The isomeric forms of butene and their corresponding conformers~\cite{lekishvili1997characterization}. One can use complex edges weights to {encode} the relative positions of different parts of a molecule. 
    Thick lines between two nodes indicate edges of weight $2$, and thin {dashed} lines indicate edges of weight $1$.
    (d,e) In the mathematical description of communication and social networks, one can use complex edge weights to distinguish between in-edges and out-edges~\cite{sanchez_infrastructure,hoser2005eigenspectral}. In (d), the blue nodes signify electric entities and the orange nodes signify communication entities.
    [We reproduced the left picture in panel (a) from \cite{schreiber2013quantum}. We reproduced panel (b) from \cite{wikimedia_3busnetwork}.]
    }
    \label{fig:examples}
\end{figure}

For each weight matrix $\mathbf{W}$, we compute the eigenvectors of $\mathbf{W} e^{\ii\varphi}$ (with $\varphi \in [0,2\pi)$) that are associated with the largest-magnitude eigenvalue. We can thereby determine the values of $\varphi$ (if there are any) for which this eigenvalue is real and positive and if the corresponding eigenvector has entries with non-negative or strictly positive real parts. That is, we can determine when $\mathbf{W}$ satisfies generalized PF properties.
We present examples of weight matrices that satisfy the strong PF property, the strong complex PF property, and none of the generalized PF properties in Table~\ref{tab:generalized_PF}. 
 Even when $\mathbf{W}$ does not satisfy any of these PF properties, we observe in our examples that there are eigenvectors whose structures are consistent with the intuitive node orderings.
\subsection{Quantum Physics}
\begin{figure}
    \centering
    \includegraphics{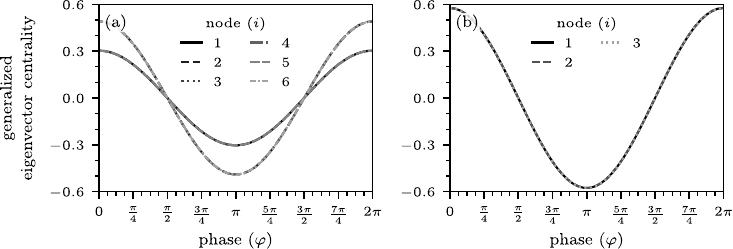}
    \caption{Generalized eigenvector centralities \eqref{eq:eigen-cent-re} in (a) a quantum scattering problem on a network with the weight matrix~\eqref{eq:tight_binding} and (b) a circuit network with the weight matrix~\eqref{eq:circuit}. We multiply the weight matrices \eqref{eq:tight_binding} and \eqref{eq:circuit} by the phase factor $\exp(\ii \varphi)$ and plot eigenvector centrality as a function of the phase $\varphi$. For the weight matrix~\eqref{eq:tight_binding}, we set $E = 1$ and $\varphi = \pi/4$.}
    \label{fig:physics_networks}
\end{figure}
In quantum physics, complex edge weights arise in the descriptions of quantum walks~\cite{bottcher2024complex}, scattering processes~\cite{wu1991quantum,liu1999electronic,Lieb2004,vasilopoulos2007aharonov}, quantum-particle statistics~\cite{nakamura2020direct,abelian_anyons2020,macikazek2019non}, and other phenomena.

As an example, we consider the tight-binding Schr\"odinger equation for single-electron dynamics on a truncated Sierpinski-gasket lattice with six lattice sites~\cite{liu1999electronic}.

The single-electron wave functions $\psi_i$ at the $i$th site satisfy the tight-binding eigenvalue equation
\begin{equation}
    \begin{pmatrix}
        E & e^{-\ii\phi} & 0 & 0 & 0 & e^{\ii\phi} \\
        e^{\ii\phi} & E & e^{-\ii\phi} & e^{\ii\phi} & 0 & e^{-\ii\phi} \\
        0 & e^{\ii\phi} & E & e^{-\ii\phi} & 0 & 0 \\
        0 & e^{-\ii\phi} & e^{\ii\phi} & E & e^{-\ii\phi} & e^{\ii\phi} \\
        0 & 0 & 0 & e^{\ii\phi} & E & e^{-\ii\phi} \\
        e^{-\ii\phi} & e^{\ii\phi} & 0 & e^{-\ii\phi} & e^{\ii\phi} & E
    \end{pmatrix}
    \begin{pmatrix}
        \psi_1 \\ \psi_2 \\ \psi_3 \\ \psi_4 \\ \psi_5 \\ \psi_6
    \end{pmatrix}
= 0\,,
    \label{eq:tight_binding}
\end{equation}
where $E$ is the electron energy and $\phi$ is the magnetic phase.

The weight matrix in Eq.~\eqref{eq:tight_binding} is Hermitian, so all of its eigenvalues are real. For $E = 1$ and $\phi = \pi/4$, its spectral radius is $3.29$ and equals its largest positive eigenvalue $\lambda_1$. 
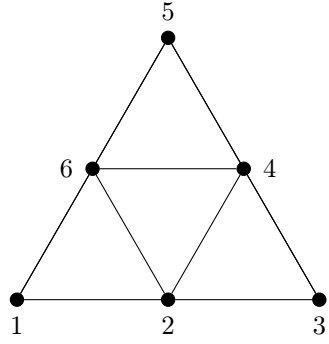
\begin{figure}
    \centering
    \begin{tikzpicture}[scale=2, every node/.style={draw, circle, fill=black, inner sep=0pt, minimum size=5pt}]

\tikzset{label distance=1mm}

    \node[label=below:1] (A) at (0,0) {};
    \node[label=below:2] (B) at (1,0) {};
    \node[label=below:3] (C) at (2,0) {};
    \node[label=left:6] (D) at (0.5,0.866) {};
    \node[label=right:4] (E) at (1.5,0.866) {};
    \node[label=above:5] (F) at (1,1.732) {};

    \draw (A) -- (B);
    \draw (B) -- (C);
    \draw (C) -- (F);
    \draw (F) -- (A);
    \draw (A) -- (D);
    \draw (D) -- (B);
    \draw (B) -- (E);
    \draw (E) -- (C);
    \draw (F) -- (D);
    \draw (D) -- (E);
    \draw (E) -- (F);

\end{tikzpicture}
    \caption{A 6-node truncated Sierpinski-gasket lattice with complex edge weights.}
    \vspace{-10pt}
    \label{fig:sierpinski}
\end{figure}

The associated eigenvector has strictly positive entries, so the weight matrix in Eq.~\eqref{eq:tight_binding} satisfies the strong PF property. 
It has entries with the value $0.30$ at nodes $1$, $3$, and $5$ and entries with the value $0.49$ at nodes $2$, $4$, and $6$. 
If we instead consider the weight matrix in Eq.~\eqref{eq:tight_binding} with $E = 1$ and $\phi = 0$, the eigenvector entries are the same, but the spectral radius is now $4.24$ instead of $3.29$.

In Figure~\ref{fig:physics_networks}, we plot the generalized eigenvector centrality~\eqref{eq:eigen-cent-re} (\ie, the real parts of the entries of the phase-dependent eigenvector) for $E = 1$ and $\phi = \pi/4$. 
Phases in the range $0\leq \varphi \lessapprox \pi/2$ yield a consistent node ordering, and phases in the range $3\pi/2\lessapprox \varphi \leq 2\pi$ {yield the same} consistent node ordering.

Nodes 1, 3, and 5 are the outer nodes of the 6-node Sierpinski-gasket lattice (see Figure~\ref{fig:sierpinski}); nodes 2, 4, and 6 are its inner nodes. The order of node centralities that we obtain by calculating generalized eigenvector centrality aligns with the intuitive expectation that nodes 2, 4, and 6 are more central.
\subsection{Circuit Theory}
\label{sec:circuit_theory}
In applications of Kirchhoff's and Ohm's laws to circuit networks in so-called ``nodal analyis''~\cite{HaytCircuitAnalysis}, one examines circuit diagrams like the one in Figure~\ref{fig:examples}(b) to determine the voltages $U_j$ of nodes $j\in\{1,\ldots,N\}$ {with respect to the ground} by solving the node-voltage equations
\begin{equation}
    \begin{pmatrix}
Y_{11} & Y_{12}  &\cdots  & Y_{1N} \\ 
Y_{21} & Y_{22}  &\cdots  & Y_{2N} \\ 
\vdots & \vdots  &\ddots  & \vdots\\ 
Y_{N1} & Y_{N2}  &\cdots  & Y_{NN} 
\end{pmatrix}
\begin{pmatrix}
	U_1\\ 
	U_2\\ 
\vdots\\ 
	U_N
\end{pmatrix} =
\begin{pmatrix}
	I_1\\ 
	I_2\\ 
\vdots\\ 
	I_N
\end{pmatrix}\,,
\end{equation}
where $I_i$ is the net current injected at node $i$ and
\begin{equation}
    Y_{ij} = 
    \begin{cases}
      y_{i} + \sum_{k\neq i} {y_{ik}} & \mbox{if} \quad i = j \\
      -y_{ij}  & \mbox{if} \quad i \neq j\,.
    \end{cases}
\end{equation}
Let $y_{ij}$ denote the admittance between nodes $i$ and $j$. The admittance $y_{ij}$ is equal to the
sum of the admittances of all branches (\ie, paths) that connect nodes $i$ and $j$. The admittance $y_i$ between node $i$ and the ground equals the sum of the admittances of all loads (\ie, a resistor, capacitor, inductor, or combination of these circuit elements between a node and the ground) that are attached to node $i$.

For a branch with admittance $y_{ij}$ that connects nodes $i$ and $j$ , the corresponding impedance is $z_{ij} = y_{ij}^{-1}$.
In an alternating-current (AC) circuit with angular frequency $\omega$, the impedances that are associated with a resistor of resistance $R$, a capacitor of capacitance $C$, and an inductor of inductance $L$ are $R$, $(\ii \omega C)^{-1}$, and $\ii \omega L$, respectively.

The 3-node network in Figure~\ref{fig:examples}(b) has an admittance matrix of
\begin{equation}
    \mathbf{Y} = \begin{pmatrix}
 	 y_{1} + y_{12} + y_{13} & -y_{12} & -y_{13} \\
 	 -y_{12} & y_{2} + y_{12} + y_{23} & -y_{23} \\
 	 -y_{13} & -y_{23} & y_{3} + y_{13} + y_{23} \\
\end{pmatrix}\,.
\end{equation}
Observe the structural similarity between the complex-valued admittance matrix and the combinatorial graph Laplacian matrix~\cite{masuda2017random}.

Consider an AC network with angular frequency $\omega = \SI{1}{\hertz}$ that consists of three nodes that are connected pairwise (\ie, dyadically) by identical series resistor--inductor (RL) elements. Each element has resistance $\SI{1}{\ohm}$ and inductance $\SI{1}{\henry}$. 
The admittance matrix is
\begin{equation}
	\mathbf{Y} = \frac{1}{2}
		\begin{pmatrix}
        			2 & \ii - 1 & \ii - 1 \\
        			\ii - 1 & 2 & \ii - 1 \\
        			\ii - 1 & \ii - 1 & 2
		\end{pmatrix},
\label{eq:circuit}
\end{equation}
where the numerical entries are in units of siemens (S). Additionally, each node is connected to the ground through a capacitor with capacitance $\SI{1}{\farad}$,
 so the admittances $y_1$, $y_2$, $y_3$ all equal $\ii \,\mathrm{S}$.

{The weight matrix~\eqref{eq:circuit} has a complex dominant eigenvalue of $\frac{3}{2} - \frac{\ii}{2}$ and a spectral radius of $1.58$.
Accordingly, the weight matrix does not satisfy any of the PF properties in Table~\ref{tab:generalized_PF}. 
Moreover, by Proposition~\ref{pro:swich-lam1-no}, no weight matrix in the same generalized-switching equivalence class satisfies these properties.
The two-dimensional eigenspace that is associated with the dominant eigenvalue does not have a
permutation-invariant eigenvector. Therefore, no dominant
eigenvector is invariant under the network's permutation symmetry. 
However, the eigenvector that is associated with the magnitude-$1$ eigenvalue $\ii$ has equal entries and is consistent with the network's symmetry. Therefore, this eigenvector yields a meaningful notion of centrality. In Figure~\ref{fig:physics_networks}(b), we show that this choice of centrality assigns the same centrality value to all nodes for all phases $\varphi$ and is thus consistent with the intuition that these circuit elements are structurally equivalent.
\subsection{Mathematical Chemistry}
\begin{figure}
    \centering
    \includegraphics{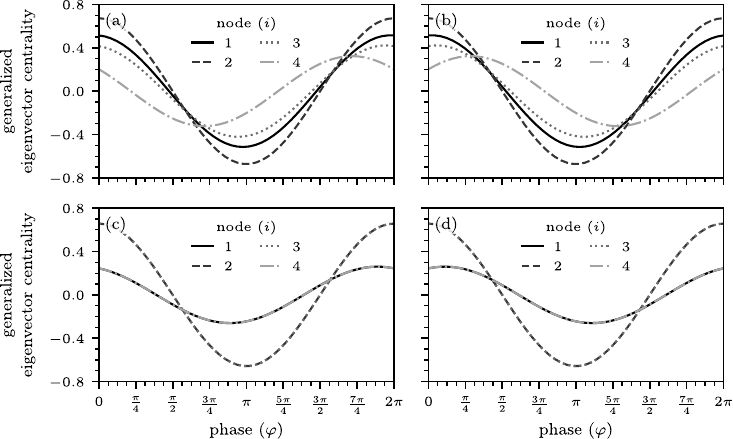}
    \caption{Generalized eigenvector centralities~\eqref{eq:eigen-cent-re} for network representations of the isomeric forms of butene and their corresponding conformers with weight matrices \eqref{eq:chemistry_network} multiplied by the phase factor $\exp(\ii\varphi)$ as functions of the phase $\varphi$.
    }
    \label{fig:chemistry_network}
\end{figure}
In mathematical chemistry, complex edge weights arise in the mathematical description of molecular structures. For example, for the isomeric forms of butene and their corresponding conformers [see Figure~\ref{fig:examples}(c)], researchers have used the weight matrices~\cite{lekishvili1997characterization}
\begin{align}
    \begin{pmatrix} 
    0 & 2 & 0 & \ii \\ 
    2 & 0 & 1 & 0 \\ 
    0 & 1 & 0 & 1 \\
    \ii & 0 & 1 & 0
    \end{pmatrix},\,\,
    \begin{pmatrix} 
    0 & 2 & 0 & -\ii \\ 
    2 & 0 & 1 & 0 \\ 
    0 & 1 & 0 & 1 \\
    -\ii & 0 & 1 & 0
    \end{pmatrix}, \,\,
    \begin{pmatrix} 
    0 & 1 & 0 & \ii \\ 
    1 & 0 & 2 & 0 \\ 
    0 & 2 & 0 & 1 \\
    \ii & 0 & 1 & 0 
    \end{pmatrix},\,\,
    \begin{pmatrix} 
    0 & 1 & 0 & -\ii \\ 
    1 & 0 & 2 & 0 \\ 
    0 & 2 & 0 & 1 \\
    -\ii & 0 & 1 & 0 
    \end{pmatrix} \,.
    \label{eq:chemistry_network}
\end{align}
An edge weight of 1 represents a single bond, and an edge weight of 2 represents a double bond. An edge weight of $\ii$ between nodes 1 and 4 indicates that these nodes are on the same side of the associated network's central edge, and an edge weight of $-\ii$ indicates that they are on opposite sides [see Figure~\ref{fig:examples}(c)]. From left to right, the spectral radii of the four weight matrices are 2.16, 2.16, 2.38, and 2.38.
For each weight matrix, two distinct eigenvalues equal the spectral radius. The eigenvalues are complex, so none of the weight matrices in~\eqref{eq:chemistry_network} satisfies any of the generalized PF properties in Table~\ref{tab:generalized_PF}. As in our example in Section~\ref{sec:circuit_theory}, Proposition~\ref{pro:swich-lam1-no} implies that no weight matrix in the same generalized-switching equivalence class satisfies these properties either. However, one can still select meaningful eigenvectors that are associated with largest-magnitude eigenvalues. Specifically, of the eigenvectors that correspond to the largest-magnitude eigenvalues, we select the one whose entries have positive real parts for $\varphi = 0$. We use this eigenvector to define $\varphi$-dependent eigenvector centralities.

Networks (i) and (ii) in Figure~\ref{fig:examples}(c) suggest an intuitive ordering of node importances, with node $2$ as the most central node and then nodes $1$, $3$, and finally $4$ for both networks. The corresponding weight matrices are, respectively, the first and second matrices in Eq.~\eqref{eq:chemistry_network}. 
Consistent with our intuition, the generalized eigenvector centralities~\eqref{eq:eigen-cent-re} in Figure~\ref{fig:chemistry_network}(a,b) rank the nodes from most central to least central in the order $2$, $1$, $3$, and $4$ for phases in the ranges $0 \le \varphi \lessapprox \pi/4$ and $7\pi/4 \lessapprox \varphi \le 2\pi$.

For networks (iii) and (iv) in Figure~\ref{fig:examples}(c), the generalized eigenvector centralities in Figure~\ref{fig:chemistry_network}(c,d) rank nodes $2$ and $3$ as more central than nodes $1$ and $4$ for phases in the ranges 
$0 \le \varphi \lessapprox \pi/4$ and $7\pi/4 \lessapprox \varphi \le 2\pi$. The corresponding weight matrices are, respectively, the third and fourth matrices in Eq.~\eqref{eq:chemistry_network}. Our eigenvector-based node ordering again aligns with the intuitive ordering from the network visualizations in Figure~\ref{fig:examples}(c).
\subsection{Social and Communication Networks}
\label{sec:social_comm_networks}
\begin{figure}
    \centering
    \includegraphics{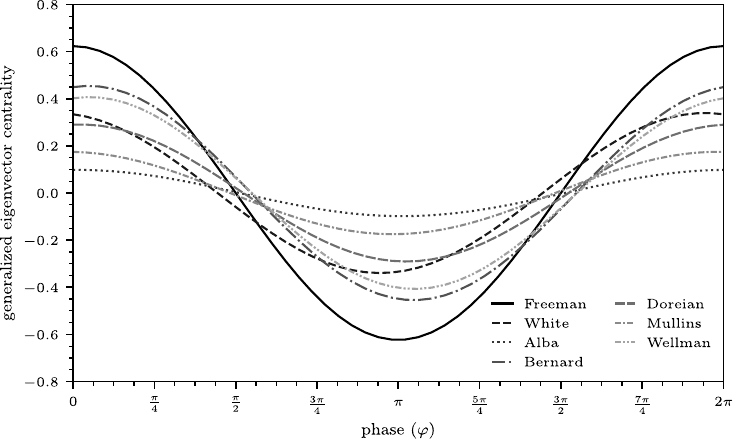}
    \caption{Generalized eigenvector centralities~\eqref{eq:eigen-cent-re} of the communication network [see Figure~\ref{fig:examples}(e)] with weight matrix~\eqref{eq:communication_network} multiplied by the phase factor $\exp(\ii\varphi)$ as functions of the phase $\varphi$.}
    \label{fig:communication_network}
\end{figure}
{
As an example of a communication network, we study a subset of the interactions in the Electronic Information Exchange System (EIES)~\cite{freeman1997uncovering} [see Figure~\ref{fig:examples}(e)]. In accordance with \cite{hoser2005eigenspectral}, we consider the weight matrix
\begin{equation}
  {\small  \mathbf{W} = \begin{pmatrix}
    0 & 115 + 84\ii & 17 + 16\ii & 93 + 127\ii & 53 + 57\ii & 33 + 23\ii & 84 + 118\ii \\
    84 + 115\ii & 0 & 4 + 10\ii & 5 + 22\ii & 5 + 9\ii & 4\ii & 15 + 24\ii \\
    16 + 17\ii & 10 + 4\ii & 0 & 15 + 17\ii & 3 + 4\ii & 3 + 3\ii & 4 + 5\ii \\
    127 + 93\ii & 22 + 5\ii & 17 + 15\ii & 0 & 57 + 57\ii & 12 + 9\ii & 34 + 35\ii \\
    57 + 53\ii & 9 + 5\ii & 4 + 3\ii & 57 + 57\ii & 0 & 8 + 8\ii & 10 + 15\ii \\
    23 + 33\ii & 4 & 3 + 3\ii & 9 + 12\ii & 8 + 8\ii & 0 & 33 + 45\ii \\
    118 + 84\ii & 24 + 15\ii & 5 + 4\ii & 35 + 34\ii & 15 + 10\ii & 45 + 33\ii & 0
    \end{pmatrix}}\,,
    \label{eq:communication_network}
\end{equation}
where the entries $w_{ij} = x_{ij} + \ii y_{ij}$ of $\mathbf{W}$ equal the sum of the number $x_{ij}$ of outbound messages from node $i$ to node $j$ and the number $y_{ij}$ of inbound messages from node $j$ to node $i$. A related example (which uses the same definitions of in-edges and out-edges) of a network with an electric layer and a communication layer was described in~\cite{sanchez_infrastructure}. An equivalent directed, real-valued representation of~\eqref{eq:communication_network} requires two weight matrices and permits only a limited spectral interpretation. Conveniently, a formulation with complex edge weights allows one to use a single matrix that one can transform into a Hermitian matrix with real eigenvalues and orthogonal eigenvectors, thereby enabling a readily interpretable spectral analysis of this asymmetric communication network.\footnote{A related approach is to encode
edge directionality via complex phases to obtain a Hermitian matrix representation of a directed network~\cite{DBLP:conf/pkdd/FurutaniSAHA19, huang2026}.}

The spectral radius of the weight matrix~\eqref{eq:communication_network} is 338.03. In Figure~\ref{fig:communication_network}, we show the generalized eigenvector centralities \eqref{eq:eigen-cent-re} of~\eqref{eq:communication_network} as functions of the phase $\varphi$. For phases in the ranges $0\leq \varphi \lessapprox  \pi/4$ and $7\pi/4\lessapprox \varphi \leq 2\pi$, the generalized eigenvector centralities are consistent with the intuitive ordering from the network visualization in Figure~\ref{fig:examples}(e). For example, the most central node for communication with the other individuals in the examined EIES subnetwork is Linton Freeman, the second-most central node is H.\ Russell Bernard, and the third-most center node is Barry Wellmann. These three scholars are extremely prominent figures in social-network analysis.

The entries of the weight matrix~\eqref{eq:communication_network} satisfy $w_{ij} = \ii \bar{w}_{ji}$, where $\bar{w}_{ji}$ is the complex conjugate of $w_{ij}$. One can transform {a weight matrix whose entries satisfy $w_{ij} = \ii \bar{w}_{ji}$}
into a Hermitian matrix by multiplying it by the phase factor $\exp(-\ii \pi/4)$~\cite{hoser2005eigenspectral}. The resulting {transformed version of the weight matrix~\eqref{eq:communication_network}} has real eigenvalues, with the largest eigenvalue equal to $338.03$. The real parts of all dominant-eigenvector entries are positive, so the transformed weight matrix satisfies the strong complex PF property. In the Supplemental Material, we demonstrate for any network with $N = 2$ nodes whose weight matrix satisfies $w_{ij} = \ii\bar{w}_{ji}$ that the transformed Hermitian matrix that one obtains by multiplication with $\exp(-\ii\pi/4)$ satisfies the complex PF property. We also give a counterexample for $N = 3$ and show that the complex PF property holds for sufficiently small skew-symmetric perturbations for $N\geq 3$.
\begin{figure}
    \centering
    \includegraphics{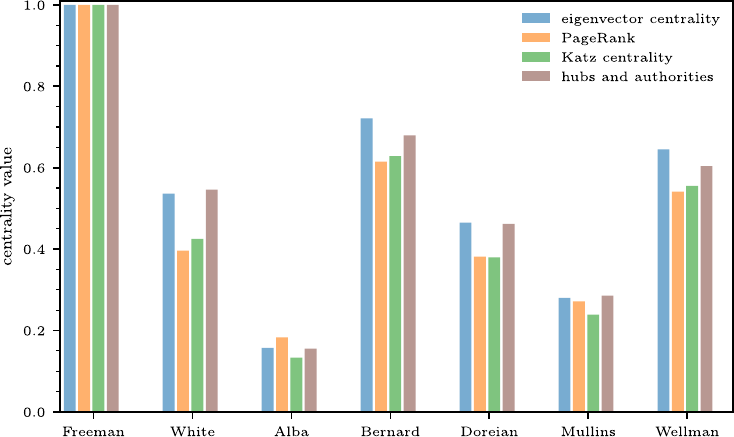}
    \caption{Comparison of different centrality measures for the communication network [see Figure~\ref{fig:examples}(e)] with weight matrix~\eqref{eq:communication_network}. We show generalized eigenvector centrality, generalized PageRank (with teleportation parameter {$a = 0.85$}), generalized Katz centrality (with downweighting factor {$\alpha = 10^{-3}$}), and generalized hub and authority centralities {(with parameter $\gamma = 1$)}.
    In all of these centrality computations, the phase parameter is $\varphi = 0$. We obtain centrality values using the real parts of generalized centralities in \eqref{eq:general_ev}, \eqref{eq:gen_hubs_auth}, \eqref{eq:pagerank}, and \eqref{eq:gen_katz}. 
         We normalize these centralities by their maxima to lie in the interval~$[0,1]$.
    }
    \label{fig:centrality_values}
\end{figure}
In Figure~\ref{fig:centrality_values}, we show generalized eigenvector centrality, generalized PageRank (with teleportation parameter {$a = 0.85$}), generalized Katz centrality (with downweighting factor {$\alpha = 10^{-3}$}), and generalized hub and authority centralities {(with parameter $\gamma = 1$)}.
In all of these centrality computations, the phase parameter is $\varphi = 0$.
We normalize the generalized centralities values by their respective maxima so that they lie in the interval~$[0,1]$. 
The examined centralities produce consistent node rankings. For the weight matrix $\mathbf{W}$ in~\eqref{eq:communication_network}, the generalized hub centrality $\mathbf{c}_{\rm H,c}$ and the generalized authority centrality $\mathbf{c}_{\rm A, c}$ have the same real part. Replacing $\mathbf{W}\mathbf{W}^\top$ and $\mathbf{W}^\top\mathbf{W}$ by $\mathbf{W}\mathbf{W}^\dagger$ and $\mathbf{W}^\dagger\mathbf{W}$, respectively, for generalized hub and authority centralities yields normalized values that {are similar to} those from generalized eigenvector centrality.
\section{Conclusions and Discussion}
\label{sec:discussion}
Perron--Frobenius (PF) theory is a cornerstone in the study of weighted networks. It underpins centrality measures such as eigenvector centrality, PageRank, and hubs and authorities. 
However, the traditional PF theorem, which was developed in the early 1900s by Georg Ferdinand Frobenius~\cite{frobenius1912} and was based on earlier work by Oskar Perron~\cite{perron1907}, applies only to non-negative matrices and not to complex-valued weight matrices, which arise in quantum information, quantum chemistry, electrodynamics, machine learning, and other fields.

In the present paper, we reviewed several generalizations of the traditional PF theorem (see Table~\ref{tab:generalized_PF}), such as ones by Noutsos and Varga~\cite{noutsos2012perron} and one by Rump~\cite{rump2003perron}. 
The computational demands of Rump's approach motivated us to focus primarily on the approaches of Noutsos and Varga.

One of our key insights is that a generalized-switching equivalence class with a weight matrix that satisfies the complex PF property yields a set of weight matrices that satisfy the complex PF property. Moreover, the equivalence class must also have a weight matrix that satisfies the PF property (which is stricter than the complex PF property). The same statements hold when we replace the complex PF and PF properties by their strong complex PF and strong PF counterparts, respectively.

We also generalized eigenvector centrality and other eigenvector-based centrality measures to networks with complex edge weights, and we calculated them on real-world networks from quantum physics, circuit theory, mathematical chemistry, and social and communication networks.

Given the broad relevance of generalized PF theory, there are many interesting avenues for future research. One relevant direction is to extend our approach to locally biased centrality measures, such as personalized PageRank~\cite{gleich2015pagerank,anderson2006,jeub2015}, in which the teleportation strategy depends on the starting location of a random walker. It is also worthwhile to generalize our approach to temporal and multilayer centrality measures~\cite{taylor2017,taylor2021tunable, taylor2023}, including variants of PageRank that can incorporate both node teleportation and layer teleportation. 
Another promising direction is to extend our results to non-normal matrices and their pseudospectra~\cite{trefethen2020spectra}. 
The pseudospectrum of a weight matrix $\mathbf{W}$ captures the eigenvalues of matrices that one obtains by small perturbations of $\mathbf{W}$.
Studying the eigenvectors of these nearby matrices may yield additional candidate centrality vectors, and some of them may satisfy generalized PF-type properties.
Another promising direction is to investigate the spectral properties of other matrices, such as graph Laplacians, that are associated with networks~\cite{zhang2011laplacian}. Yu et al.~\cite{yu2025directedgraphsreallaplacian} recently identified both structural properties of directed networks that guarantee a purely real combinatorial-Laplacian spectrum and structural properties that yield complex eigenvalues.
It is also useful to  closely examine the uniqueness of eigenvectors that are associated with PF generalizations and to study the relevance of bulk eigenvectors (\ie, eigenvectors that are not associated with extremal eigenvalues) in network analysis~\cite{cucuringu2011localization,masuda2022dimension,bottcher2026graph}. 
\section*{Acknowledgements}
We thank Rajasekhar Anguluri, Gregory Hemenway, Aditi Saxena, and Twinkle Tripathy for helpful comments.
%


\bibliographystyle{siamplain}
\bibliography{refs-v10.bib}

\clearpage


\section*{Supplementary material}

**1.3. Proof of Proposition 3.7.** Consider a weight matrix $\mathbf{W}$ with a dominant eigenvalue $\lambda_1 > 0$, and let $\mathbf{x}$ be an associated eigenvector of this eigenvalue. One can write each eigenvector entry $x_i$ of $\mathbf{x}$ in polar form as $x_i = r_i e^{\mathrm{i}\theta(x_i)}$, so we define a generalized-switching function $\tau$ by $\tau(i) = e^{-\mathrm{i}\theta(x_i)}$, where $\theta(x_i)$ is the phase of $x_i$ if $x_i \neq 0$ and $\theta(x_i) = 0$ if $x_i = 0$. We now apply the generalized-switching operation $\tau$ to obtain the weight matrix $\mathbf{W}' = \mathbf{D}(\tau)\mathbf{W}\mathbf{D}(\tau)^{-1}$, which has the same spectrum as $\mathbf{W}$. Therefore, $\lambda_1 > 0$ is also the dominant eigenvalue of $\mathbf{W}'$. The $i$th entry of the associated eigenvector $\mathbf{D}(\tau)\mathbf{x}$ is $\tau(i)x_i = e^{-\mathrm{i}\theta(x_i)}x_i = r_i \geq 0$. Therefore, $\mathbf{W}'$ has a dominant eigenvalue with an associated eigenvector whose entries are real and non-negative, so $\mathbf{W}'$ satisfies the PF property. Furthermore, for any generalized-switching function $\tau$ with $\mathrm{Re}\big(\tau(i)x_i\big) \geq 0$ for all $i$, the matrix $\mathbf{W}'$ has a dominant eigenvalue with an associated eigenvector whose entries have non-negative real parts. The matrix $\mathbf{W}'$ thus satisfies the complex PF property.

If the dominant eigenvalue of $\mathbf{W}$ is simple and its associated eigenvector does not have any 0 entries, then $r_i > 0$ for all $i$. In this case, the generalized-switching function $\tau$ with $\tau(i) = e^{-\mathrm{i}\theta(x_i)}$ yields an eigenvector of $\mathbf{W}'$ whose entries are real and strictly positive. Consequently, $\mathbf{W}'$ satisfies the strong PF property. Additionally, for any generalized-switching function $\tau$ with $\mathrm{Re}\big(\tau(i)x_i\big) > 0$ for all $i$, the matrix $\mathbf{W}'$ satisfies the strong complex PF property. ∎

**2. Generalized PF properties of weight matrices with entries $w_{ij} = \mathrm{i}\,\overline{w}_{ji}$.** In this section, we examine which weight matrices with entries $w_{ij} = \mathrm{i}\,\overline{w}_{ji}$, which arose in the description of social and communication networks in Section 5.4 in the main manuscript, satisfy the complex PF property.

**2.1. Complex PF property for graphs with $N = 2$ nodes.** The case $N = 2$ is straightforward. As we discussed in the main manuscript (see Section 4), a sufficient condition for the existence of an eigenvector whose entries all have non-negative real parts is that the phases of its components lie within a closed interval of length $\pi$. (That is, all eigenvalues lie within a half-disk.)

THEOREM 2.1. *Let $\widetilde{\mathbf{W}} = \exp(-\mathrm{i}\pi/4)\mathbf{W}$, where $\mathbf{W} = \mathbf{X} + \mathrm{i}\mathbf{X}^\top$ and $\mathbf{X}$ is a matrix with non-negative entries with $X_{ii} = 0$ for all $i$. If $N = 2$ and the graph $G$ with edge set $\{(i,j)\colon \big|\widetilde{W}_{ij}\big| > 0\}$ is strongly connected, then $\widetilde{\mathbf{W}}$ has a dominant eigenvector whose entries all have non-negative real parts.*

*Proof.* Because $N = 2$ and $G$ is strongly connected, we necessarily have $r := \big|\widetilde{W}_{12}\big| > 0$. One can write the matrix in the form

$$\widetilde{\mathbf{W}} = \begin{pmatrix} 0 & re^{\mathrm{i}\theta} \\ re^{-\mathrm{i}\theta} & 0 \end{pmatrix},$$

where $\theta \in [-\pi/4, \pi/4]$.

The eigenvalues of $\widetilde{\mathbf{W}}$ are $\pm r$. An eigenvector that is associated with the eigenvalue $r$ is

$$\mathbf{v} = \begin{pmatrix} 1 \\ e^{-\mathrm{i}\theta} \end{pmatrix}.$$

Because $\theta \in [-\pi/4, \pi/4]$, we have $\mathrm{Re}(1) = 1 \geq 0$ and $\mathrm{Re}(e^{-\mathrm{i}\theta}) = \cos\theta \geq 0$. Therefore, all entries of $\mathbf{v}$ have non-negative real parts and the matrix $\widetilde{\mathbf{W}}$ satisfies the complex PF property. ∎

**2.2. Counterexample for graphs with $N = 3$ nodes.** We use the matrix

$$\mathbf{X} = \begin{pmatrix} 0 & 1 & 10 \\ 12 & 0 & 1 \\ 1 & 12 & 0 \end{pmatrix} \tag{2.1}$$

to construct a weight matrix $\mathbf{W} = \mathbf{X} + \mathrm{i}\mathbf{X}^\top$ of a 3-node graph. We then rotate $\mathbf{W}$ by multiplying it by the phase factor $\exp(-\mathrm{i}\pi/4)$ to obtain $\widetilde{\mathbf{W}} = \exp(-\mathrm{i}\pi/4)\mathbf{W}$.

From a numerical computation, we find that the largest-magnitude eigenvalue of $\widetilde{\mathbf{W}}$ is approximately $-21.414$. Therefore, $\rho(\widetilde{\mathbf{W}}) \approx 21.414$, but none of its positive eigenvalues attain this value. Therefore, $\widetilde{\mathbf{W}}$ does not satisfy any of the PF properties in Table 1 of the main manuscript. For completeness, we briefly discuss the vector

$$\mathbf{v} = \begin{pmatrix} -0.564 \\ 0.315 - 0.514\mathrm{i} \\ 0.257 + 0.503\mathrm{i} \end{pmatrix}, \tag{2.2}$$

which is a normalized eigenvector that is associated with this eigenvalue. The phases (in radians) of this eigenvector's entries are

$$\{3.142, -1.021, 1.099\}. \tag{2.3}$$

These phases are not contained simultaneously in any interval of length $\pi$. The smallest interval that includes all of these phases has a length of approximately $4.163 > \pi$. Equivalently, there is no global phase factor $e^{\mathrm{i}\varphi}$ such that all entries of $e^{\mathrm{i}\varphi}\mathbf{v}$ have non-negative real parts.

**2.3. A perturbation-based condition for the complex PF property.** The matrix $\widetilde{\mathbf{W}} = \exp(-\mathrm{i}\pi/4)\mathbf{W}$, where $\mathbf{W} = \mathbf{X} + \mathrm{i}\mathbf{X}^\top$ and $\mathbf{X}$ is non-negative, is Hermitian, so we can decompose it as

$$\widetilde{\mathbf{W}} = \mathbf{S} + \mathrm{i}\mathbf{K}, \tag{2.4}$$

where $\mathbf{S}$ is a real symmetric matrix and $\mathbf{K}$ is a real skew-symmetric matrix. A direct computation yields

$$S_{ij} = \frac{1}{\sqrt{2}}(X_{ij} + X_{ji}) \geq 0, \tag{2.5}$$

so $\mathbf{S}$ is entrywise non-negative. If the graph that corresponds to $\mathbf{S}$ is connected, then $\mathbf{S}$ is irreducible. Therefore, by the traditional PF theorem, the largest eigenvalue $\lambda_1 = \rho(\mathbf{S})$ is simple and has an associated unit eigenvector $\mathbf{u}$ with strictly positive entries.

We now show that $\widetilde{\mathbf{W}}$ satisfies the complex PF property for sufficiently small skew-symmetric perturbations $\mathbf{K}$ in Eq. (2.4).

THEOREM 2.2. *Assume that the graph with weight matrix $\mathbf{S}$ is connected. Let $\lambda_1 > \lambda_2$ be the two largest eigenvalues of $\mathbf{S}$, and let $\Delta = \lambda_1 - \lambda_2$. Let $\mathbf{u}$ be the unit PF eigenvector of $\mathbf{S}$ (which implies that $\mathbf{u} > 0$ entrywise), and define $u_{\min} = \min_i u_i$. If*

$$\frac{\|\mathbf{K}\|_2}{\Delta} \leq \frac{u_{\min}}{\sqrt{1 + u_{\min}^2}}, \tag{2.6}$$

*then one can choose a dominant eigenvector of the Hermitian matrix $\widetilde{\mathbf{W}}$ that has a non-negative real part.*

*Proof.* Let $\mathbf{E} \coloneqq \mathrm{i}\mathbf{K}$, so $\widetilde{\mathbf{W}} = \mathbf{S} + \mathbf{E}$ and $\|\mathbf{E}\|_2 = \|\mathbf{K}\|_2$. The matrix $\mathbf{S}$ is entrywise non-negative and irreducible, so (by the traditional PF theorem for non-negative matrices) the eigenvalue $\lambda_1$ is simple and its corresponding unit eigenvector $\mathbf{u}$ has strictly positive entries.

Let $\mathbf{v}$ be a unit eigenvector that corresponds to the dominant eigenvalue of $\widetilde{\mathbf{W}}$. We adjust the global complex phase of $\mathbf{v}$ so that $\langle \mathbf{u}, \mathbf{v} \rangle \geq 0$. Applying the Davis–Kahan theorem [1] for Hermitian matrices and one-dimensional invariant subspaces gives

$$\sin\theta \leq \frac{\|\mathbf{E}\|_2}{\Delta} \,,$$

where $\theta = \angle(\mathrm{span}\{\mathbf{u}\}, \mathrm{span}\{\mathbf{v}\}) \in [0, \pi/2]$ denotes the principal angle between the subspaces. By definition, the principle angle satisfies $\cos\theta = |\langle \mathbf{u}, \mathbf{v} \rangle|$. Additionally, let $\eta = \|\mathbf{E}\|_2/\Delta$. Therefore, there exist $c \geq 0$ and $\mathbf{w} \in \mathbb{C}^n$ with $\langle \mathbf{u}, \mathbf{w} \rangle = 0$ such that

$$\mathbf{v} = c\,\mathbf{u} + \mathbf{w} \,, \quad c = \cos\theta \geq \sqrt{1 - \eta^2} \,, \quad \|\mathbf{w}\|_2 = \sin\theta \leq \eta \,.$$

The vector $\mathbf{u}$ is real, so

$$\mathrm{Re}(\mathbf{v}) = c\,\mathbf{u} + \mathrm{Re}(\mathbf{w}) \,.$$

Therefore, for each $i$, we have

$$\begin{aligned} \mathrm{Re}(v_i) &\geq c\,u_i - |\,\mathrm{Re}(w_i)| \\ &\geq c\,u_i - |w_i| \\ &\geq c\,u_i - \|\mathbf{w}\|_2 \\ &\geq \sqrt{1 - \eta^2}\,u_i - \eta \,. \end{aligned}$$

Consequently, $\mathrm{Re}(v_i) \geq 0$ for all $i$ whenever

$$\sqrt{1 - \eta^2}\,u_{\min} \geq \eta \,. \tag{2.7}$$

This completes the proof. ∎

For the counterexample (2.1), we have $\Delta \approx 25.236$, $u_{\min} \approx 0.567$, and $\|\mathbf{K}\|_2 \approx 12.708$. Therefore,

$$\frac{\|\mathbf{K}\|_2}{\Delta} \approx 0.504 \quad \text{and} \quad \frac{u_{\min}}{\sqrt{1 + u_{\min}^2}} \approx 0.493 \,. \tag{2.8}$$

Because $0.504 > 0.493$, the inequality (2.6) does not hold, which is consistent with the fact that the matrix (2.1) does not satisfy the complex PF property.



\end{document}